\documentclass[a4paper]{amsart}
\usepackage{graphicx, stmaryrd} % Required for inserting images
\usepackage{enumerate}  
\usepackage{hyperref}   % www references
\usepackage[square,numbers]{natbib}     % bibliography
\usepackage{amsthm}     % ams theorem environments
\usepackage{amsmath}    % ams math environments
\usepackage{amssymb}    % ams symbols
\usepackage{amsxtra}    % ams extra symbols
\usepackage{eucal}      % eu calligraphic fonts
\usepackage[bbgreekl]{mathbbol}   % black board symbols
\usepackage{todonotes}
\usepackage{comment}
\usepackage{tikz}

\usetikzlibrary{matrix}
\usetikzlibrary{patterns, shapes}
\usetikzlibrary{arrows}
\usetikzlibrary{calc,3d}
\usetikzlibrary{decorations,decorations.pathmorphing, decorations.pathreplacing}
\usetikzlibrary{through}
\tikzset{every picture/.style={baseline=-.65ex}}
\tikzset{ext/.style={circle, draw,inner sep=1pt},int/.style={circle,draw,fill,inner sep=1pt},nil/.style={inner sep=1pt}}
\tikzset{exte/.style={circle, draw,inner sep=3pt},inte/.style={circle,draw,fill,inner sep=3pt}}
\tikzset{diagram/.style={matrix of math nodes, row sep=3em, column sep=2.5em, text height=1.5ex, text depth=0.25ex}}
\tikzset{diagram2/.style={matrix of math nodes, row sep=0.5em, column sep=0.5em, text height=1.5ex, text depth=0.25ex}}
\tikzset{every loop/.style={draw}}

\newcommand{\tadpole}{
\begin{tikzpicture}[baseline=-.65ex]
\node[int] (v) at (0,0) {};
\draw (v) edge[loop] (v);
\end{tikzpicture}
}

\usepackage{tikz-cd}

\theoremstyle{plain}

\newtheorem{introthm}{Theorem}

\newtheorem{thm}{Theorem}[section]

\newtheorem{prop}[thm]{Proposition}
\newtheorem{proposition}[thm]{Proposition}
\newtheorem{lemma}[thm]{Lemma}
\newtheorem{cor}[thm]{Corollary}
\theoremstyle{definition}
\newtheorem{rem}[thm]{Remark}
\newtheorem{definition}[thm]{Definition}
\newtheorem{ex}[thm]{Example}
\newcommand{\G}{\mathsf{G}}
\newcommand{\St}{\mathrm{St}}
\newcommand{\Q}{\mathbb{Q}}

\newcommand{\C}{\mathbb{C}}
\newcommand{\CP}{\mathbb{CP}}
\newcommand{\T}{\mathcal T}
\newcommand{\fc}{\mathfrak c}
\newcommand{\GonL}{\mathcal L^G}
\newcommand{\fG}{\mathsf{fG}}
\newcommand{\fGc}{\mathsf{fcG}}
\newcommand{\fGC}{\mathsf{fGC}}
\newcommand{\fGCc}{\mathsf{fcGC}}
\newcommand{\GC}{\mathsf{GC}}
\newcommand{\mF}{\mathcal{F}}
\DeclareMathOperator{\gr}{gr}
\DeclareMathOperator{\vdim}{dim}

\DeclareMathOperator{\ad}{ad}

\newcommand{\GL}{\mathrm{GL}}
\newcommand{\SL}{\mathrm{SL}}
\newcommand{\Z}{\mathbb{Z}}
\newcommand{\fg}{\mathfrak{g}}
\newcommand{\fh}{\mathfrak{h}}

\newcommand{\shB}{\mathcal B}
\newcommand{\shA}{\mathcal{H}}
\newcommand{\RR}{R}
\newcommand{\Sh}{\mathrm{Sh}}

\DeclareMathOperator{\vspan}{span}
\newcommand{\FreeLie}{\mathrm{Free}_{\mathrm{Lie}}}
\newcommand{\Proj}{\mathrm{Proj}}

\newcommand{\sat}{{\mathrm{sat}}}
\newcommand{\id}{{\mathrm{id}}}
\renewcommand{\subset}{\subseteq}
\newcommand{\grt}{\mathfrak{grt}}
\usepackage{todonotes}
\DeclareMathOperator{\coker}{coker}
\DeclareMathOperator{\im}{im}
\newcommand{\rk}{\operatorname{rk}}

\author{Peter Patzt}
\address{University of Oklahoma \\
601 Elm Avenue \\
Norman, OK 73019, USA}
\email{ppatzt@ou.edu}

\author{Thomas Willwacher}
\address{Department of Mathematics\\ ETH Zurich\\ R\"amistrasse 101 \\ 8092 Zurich, Switzerland}
\email{thomas.willwacher@math.ethz.ch}

\title[The graph complex and the homology of the general linear groups]{The graph complex and the homology of the general linear groups revisited}

\date{}

\begin{document}

\begin{abstract}
    We revisit the relation between the Kontsevich commutative graph complex and the homology of $\GL(n,\Z)$ found by Brown--Chan--Galatius--Payne using the sharbly complex of Ash--Miller--Patzt.
\end{abstract}

\maketitle

\tableofcontents

\section{Introduction}
We describe a streamlined alternative variant of a connection between the Kontsevich graph homology and that of $\GL(n,\Z)$ and $\SL(n,\Z)$ that was put forward in \cite{BCGP}, but using the simpler methods of \cite{AMP}.
Let us recall or introduce the main objects of interest. 
We are interested in the rational cohomology of the general and special linear groups $H^\bullet(\GL(n,\Z); \Q)$ and $H^\bullet(\SL(n,\Z); \Q)$. Borel--Serre duality \cite{BorelSerre} implies that these satisfy, 
\begin{align*}
    H^{{\binom{n}{2}}-k}(\SL(n,\Z);\Q) &\cong H_k(\GL(n,\Z);\St_n\otimes \Q) \oplus H_k(\GL(n,\Z);\St_n\otimes \Q\otimes \det) \\
    H^{{\binom{n}{2}}-k}(\GL(n,\Z);\Q) &\cong 
    \begin{cases}
H_k(\GL(n,\Z);\St_n\otimes \Q) & \text{for $n$ odd} \\
H_k(\GL(n,\Z);\St_n\otimes \Q\otimes \det) & \text{for $n$ even}
    \end{cases}
\end{align*}
with $\St_n = \St_n(\Q)$ being the Steinberg module of the rationals (see Section \ref{sec:Steinberg}). Consider the direct sum
\[
\shA = \bigoplus_n H_{\bullet-n}(\GL(n,\Z);\St_n\otimes \Q).
\]
Then the following result has been shown independently by Brown--Chan--Galatius--Payne and Ash--Miller--Patzt, by distinct methods.
\begin{thm}[{\cite{AMP,BCGP}}]
There is a graded commutative Hopf algebra structure on the graded vector space $\shA$.
\end{thm}
Both methods have been shown to yield the same Hopf algebra structure in \cite{PalPayne}.
We will follow here the approach of \cite{AMP}. They show the above theorem by explicitly constructing a Hopf algebra structure on a chain level model for $\shA$,
\begin{equation}\label{equ:B def into}
\shB := \bigoplus_n \Sh(\Z^n)_{\GL(n,\Z)},
\end{equation}
where the sharbly complex $\Sh(\Z^n)$ is a flat resolution of the Steinberg module $\St_n\otimes \Q$ up to a degree shift. We refer to Section \ref{sec:sharbly} for more details and the formulas describing the Hopf algebra structure. Here we just remark that elements of $\Sh_{k}(\Z^n)$ are essentially $\Q$-linear combinations of $n\times k$ integer matrices whose columns span $\Q^n$.

Another result of \cite{BCGP} is the construction of maps relating the Hopf algebra $\shA$ above to the Kontsevich graph complex $\GC_0$, and to the K-theory of $\Z$.
Analogous results are not present in \cite{AMP}.
The purpose of this paper is to fill that gap. In fact, we shall see that the sharbly formalism allows for a very clean and beautiful formulation of the link between the cohomology of $\GC_0$ and that of $\GL(n,\Z)$.

The Kontsevich graph complex $\GC_0$ is a dg Lie algebra of $\Q$-linear combinations of connected graphs with at least trivalent vertices, see Section \ref{sec:graphs} for the precise definition.
The subscript ``0'' (as opposed to the more commonly used $\GC_2$) refers to the fact that for us the cohomological grading is by number of edges of the graphs.
To state our results most conveniently, we shall need a slight variant. Namely, we pass to a larger dg Lie algebra 
\[
\GC_0^+ := \GC_0 \oplus \Q \tadpole
\]
by adding a single graph cocycle with a bivalent vertex, which we will call the \emph{tadpole graph} in this paper.
We then consider the graded dual dg Lie coalgebra $\G_0^+$ of $\GC_0^+$. Elements of $\G_0^+$ are linear combinations of graphs, but equipped with the dual differential and Lie coalgebra structure. We denote by $U^c\G_0^+$ the universal coenveloping Hopf algebra of $\G_0^+$. It is naturally a dg commutative Hopf algebra.
Then our first main result is:
\begin{introthm}\label{thm:main 1}
    The map 
    \[
    \Phi \colon U^c\G_0^+ \longrightarrow \shB
    \]
    sending a graph to its loop-edge adjacency matrix, understood as a sharbly,\footnote{See Section \ref{sec:Phi} below for a more precise definition of $\Phi$.} is a morphism of dg Hopf algebras, where the right-hand side is equipped with the Hopf algebra structure of \cite{AMP}. 
\end{introthm}

Using Brown--Schnetz \cite[Theorem 4(i)]{BrownSchnetz}, it is easy to show (see Lemma \ref{lem:wheelnonzero}) that the map $\Phi$ sends the wheel classes $W_{2k+1}\in \G_0$ to non-zero classes in $H(\shB)$. We also give a new proof using Goncharov's Grassmannian polylogarithm \cite{Goncharov} in Remark \ref{rem:wheel indep}.

Theorem \ref{thm:main 1} and its proof in fact imply the following lower bound on the homology for all PIDs $\RR \subset \C$, improving upon similar results of Brown--Chan--Galatius--Payne \cite{BCGP} and Ash--Miller--Patzt \cite{AMP}.
\begin{introthm}\label{thm:PID lower}
    Let $\RR$ be a PID with field of fractions $F$ such that an injective ring morphism $\RR\to \C$ exists.
    Let 
    \[
    \shA_{\RR} = \bigoplus_n H_{\bullet-n}(\GL(n,\RR);\St_n(F)\otimes \Q)
    \]
    be the Hopf algebra constructed by Ash--Miller--Patzt \cite{AMP}.
    Then there is a surjective morphism of bigraded (by degree and rank) Hopf algebras
    \[
    \shA_{\RR} \otimes \C \longrightarrow \Lambda_\C ( t) \otimes T_\C(s_3,s_5,\dots),
    \]
    where $T_\C (s_3,s_5,\dots)$ on the right-hand side is the tensor coalgebra in symbols $s_m$ of degree $2m$ and rank $m$ for $m\geq 3$ odd, equipped with the shuffle product, $t$ is a symbol of degree $1$ and rank $1$, and $\Lambda_\C (t)\cong \C \oplus \C t$ is the exterior algebra.
\end{introthm}

An essentially equivalent result was established independently by Brown--Panzer--Portner \cite[Theorem 1]{BrownPanzerPortner}.

Next, let $\nabla$ be the action of the special graph $\frac 12 T = \frac12 \tadpole\in \GC_0^+$ on $U^c\G_0^+$, i.e. 
\[\nabla x = -\frac12 (T \otimes \id)\Delta(x) + \frac 12 (\id \otimes T)\Delta(x)\]
for $x\in U^c\G_0^+$, where $x$ is homogeneous of degree $n$.
Graphically, the operation $\nabla$ removes an edge between two vertices, reducing the total loop order by one. (See Lemma \ref{lem:T action}.)
Then one can check that $d\nabla+\nabla d=0$ and $\nabla^2=0$, where $d$ is the differential on $U^c\G_0^+$. 
It follows that $d+\nabla$ is again another differential on $U^c\G_0^+$.
It follows from computations of \cite{KWZ} that the homology with respect to the twisted differential is
\[
H(U^c\G_0^+,d+\nabla) \cong H(BK(\Z);\Q) \cong \Q[t,W_3, W_5,W_7,\dots],
\]
where $\Q[\cdots]$ denotes the free graded commutative algebra in the given generators.  It is isomorphic to the rational homology of $BK(\Z)$.

Our second main result is:
\begin{introthm}\label{thm:main 2}
    Let $\Phi$ be the map of Theorem \ref{thm:main 1}. Then the following hold.
    
    \begin{enumerate}[(i)]
        \item There is a deformation $d+D$ of the differential $d$ on $\shB$, defined in Section \ref{sec:deformed D} below, such that $\Phi\circ (d+\nabla)=(d+D)\circ \Phi$, so that $\Phi$ defines a morphism of dg Hopf algebras 
        \begin{equation}\label{equ:Phi deformed}
        \Phi \colon (U^c\G_0^+, d+\nabla) \longrightarrow (\shB, d+D).
        \end{equation}
        \item This morphism \eqref{equ:Phi deformed} is a quasi-isomorphism of dg Hopf algebras. The homology of either side may be identified with 
        \[
        \Q[t,W_3,W_5,W_7,\dots] \cong H_\bullet(BK(\Z);\Q),
        \]
        where $t$ in $(U^c\G_0^+, d+\nabla)$ is represented by the tadpole graph of degree 1 and $W_n$ is the wheel graph with $n$ spokes of degree $2n$. 
        \end{enumerate}
\end{introthm}

Both sides of \eqref{equ:Phi deformed} hence compute the rational K-theory of $\Z$.
We remark that a version of the above results is contained in \cite{BCGP}, and we do not claim originality. Our reformulation has the advantage that it is leaner and (arguably) more natural, in that we do not need to modify the standard definition of the graph complex beyond taking a universal enveloping coalgebra. Note that the spectral sequence coming from the double complex $(\shB,d,D)$ is a spectral sequence of Hopf algebras as in \cite{BCGP}.

\subsection*{Acknowledgements}
We thank F. Brown and A. Goncharov for helpful discussions.

\subsection*{AI usage statement}
The project originated before the available AI tools were strong enough to meaningfully contribute. In the later stages ChatGPT pro was used for proofreading, discussions and literature research.

\section{Hopf algebra of sharblies }
\label{sec:sharbly} 
\subsection{The Tits building and Steinberg module}
\label{sec:Steinberg}
Let $\RR$ be a PID and $F$ its field of fractions.
The \emph{Tits building} $\T_\bullet(V)$ of an $n$-dimensional $F$-vector space is a simplicial set whose $p$-simplices are ascending chains 
\[
V_0 \subset V_1\subset \cdots \subset V_p
\]
of nonzero proper subspaces $0\subsetneq V_j \subsetneq V$.
The simplicial face maps $d_j$ ($j=0,\dots,p$) simply forget the $j$-th subspace of the chain,
\[
d_j(V_0 \subset V_1\subset \cdots \subset V_p)
=
(V_0 \subset V_1\subset \cdots \subset \hat V_j \subset \cdots \subset V_p)
\]
and the simplicial degeneracy maps $s_j$ ($j = 0, \dots, p$) double the $j$-th subspace
\[
s_j(V_0 \subset V_1\subset \cdots \subset V_p)
=
(V_0 \subset V_1\subset \cdots \subset V_j \subset V_j \subset \cdots \subset V_p).
\]

Let $n\geq 2$.
By the Solomon--Tits Theorem \cite{Solomon} the reduced homology of $\T_\bullet(V)$ is concentrated in degree $n-2$, and the \emph{Steinberg module} is defined to be this reduced (integral) homology,
\[
\St(V) := \widetilde H_{n-2}(\T_\bullet(V);\Z).
\]
In particular, we will set $\St_n(F):=\St(F^n)$.

In the special cases $n=0,1$ the Tits building $\T(F^n)$ is empty.
Nevertheless, we define the Steinberg module to be $\St_0(F)=\St_1(F)=\Z$.

\subsection{Augmented reduced complex}\label{sec:red aug}
Below, we will use the augmented reduced normalized rational chain complex $\hat N \T_\bullet(V)$ of the Tits building.
It is the same as the normalized chain complex with $\Q$-coefficients, but we add another generator $\emptyset$ in degree $-1$ that can be thought of representing the empty chain of subspaces in the Tits building.
\[
\hat N \T_\bullet(V) = N \T_\bullet(V)\oplus \Q\emptyset.
\]
The differential is extended by sending any vertex of $\T_\bullet(V)$ to $\emptyset$, so that $\hat N \T_\bullet(V)$ computes the reduced homology of $\T_\bullet(V)$, as long as $\vdim V\geq 2$.
This convention also fits well with the definition of $\St_n$ in the degenerate case $n=1$, where now $H(\hat N \T_\bullet(V))\cong \Q\emptyset\cong \St_n(F)\otimes \Q$, concentrated in degree $-1$.
\subsection{The sharbly Hopf algebra}
We recall the definition of the Hopf algebra of sharblies from \cite{AMP}. 
Let $\RR$ be a PID, $F$ its field of fractions, and let $V$ be a free $\RR$-module of rank $n$. Then we define the differential graded $\Q$-vector space $\Sh(V)$ whose elements of degree $N$ are formal $\Q$-linear combinations of tuples 
\[
[v_1,\dots, v_{N}]
\]
of vectors $v_1,\dots,v_{N}\in V$, modulo the following relations:
\begin{itemize}
    \item Antisymmetry: For $\sigma\in S_{N}$ a permutation we set 
    \begin{equation}\label{equ:sh perm rel}
    [v_1,\dots, v_{N}] = (-1)^\sigma [v_{\sigma(1)},\dots,v_{\sigma(N)}].
    \end{equation}
    \item Nullity: If any $v_j=0$, then we set 
    \begin{equation}\label{equ:sh null rel}
    [v_1,\dots, v_{N}] = 0.
    \end{equation}
    \item Span: If $v_1,\dots,v_{N}$ do not span $V \otimes_\RR F$ as an $F$-vector space, then we set 
    \begin{equation}\label{equ:sh span rel}
    [v_1,\dots, v_{N}] = 0.
    \end{equation}
    \item Sign: For any $j$
    \begin{equation}\label{equ:sh sign rel}
    [v_1,\dots,v_j,\dots, v_{N}] =[v_1,\dots,-v_j,\dots, v_{N}].
    \end{equation}
    %\todo[inline]{This relation appears also in other papers, but destroys the finite stabilizer property. }
\end{itemize}
Note in particular that for the case of $V=0$ being the null module we have $\Sh(0)\cong \Q$, spanned by the ``empty'' sharbly $[\,]$.

The differential is 
\begin{equation}\label{equ:sh diff}
d [v_1,\dots, v_{N}] = 
\sum_{j=1\atop v_j\neq 0}^N (-1)^{j-1} [v_1,\dots,\hat v_j,\dots, v_{N}].
\end{equation}
Note that the extra condition $v_j\neq 0$ in the sum is necessary for otherwise the differential would not respect the nullity relation \eqref{equ:sh null rel}. Of course, due to the same relation, we will usually only apply the differential to sharblies for which all $v_j\neq 0$. Furthermore, we call an element $v_j$ of a non-zero sharbly $[v_1,\dots,v_N]\in \Sh(V)$ a \emph{coloop} if $\vspan(v_1,\dots,\hat v_j,\dots,v_N)\neq V\otimes_R F$.
Note that removing a coloop in \eqref{equ:sh diff} results in a zero sharbly. Hence we can equivalently write
\begin{equation}\label{equ:d simplified}
d [v_1,\dots, v_{N}] = 
\sum_{j=1\atop {v_j\neq 0 \atop v_j \text{ no coloop}} }^N (-1)^{j-1} [v_1,\dots,\hat v_j,\dots, v_{N}].
\end{equation}

Lee--Szczarba \cite{LeeSzczarba} first introduced the chain complex $\Sh(V)$ and proved that its homology is isomorphic to $\St(V \otimes_\RR F)\otimes \Q$, concentrated in degree (length) $n$.

There is a natural product operation 
\begin{equation}\label{equ:sh prod}
\begin{gathered}
\Sh(V)\otimes \Sh(W) \longrightarrow \Sh(V\oplus W) \\
[v_1,\dots, v_{N}]
\cdot
[w_1,\dots, w_{M}]
=
[v_1,\dots, v_{N},w_1,\dots, w_{M}],
\end{gathered}
\end{equation}
where on the right-hand side we understand that $v_j\in V\subset V\oplus W$ etc.
 
There is, furthermore, a coproduct operation 
\begin{gather}
\Sh(V) \longrightarrow \bigoplus_{\substack{W\subset V\\W\text{ saturated}}}
\Sh(W) \otimes \Sh(V/ W) \\
[v_1,\dots, v_{N}]
\mapsto 
\sum_{m=0}^N\sum_{1\leq i_1<\cdots <i_m\leq N}\pm 
[v_{i_1},\dots, v_{i_m}] \otimes [\bar v_1,\dots, \hat v_{i_1} ,\dots \hat v_{i_m},\dots, \bar v_{N}]
,
\label{equ:sh coprod}
\end{gather}
where the sign is the sign of the unshuffle permutation corresponding to the choice of subset $I=\{i_1,\dots ,i_m\}$. The term corresponding to $I$ on the right-hand side lives in the summand corresponding to 
\[
W=\vspan_F(v_i \otimes 1 \in V \otimes_R F \mid i\in I)\cap (V \otimes_R R)=:\vspan(v_i \mid i\in I)_{\sat}, 
\]
and we denote by $\bar v_j$ for $j\notin I$ the image of $v_j$ in the quotient $V/ W$. 

There is an obvious action of $\GL(V)$ on $\Sh(V)$.
We define the space of coinvariants 
\[
\shB_\RR:= \bigoplus_{n=0}^\infty \Sh(\RR^n)_{\GL(n,\RR)}.
\]
The product and coproduct above descend to a product and coproduct on $\shB_\RR$.
\begin{thm}[Ash--Miller--Patzt \cite{AMP}]
\label{thm:AMP}
Let $\RR$ be a PID.
Then we have that 
\begin{equation}\label{equ:AMP}
H_{n+k}(\Sh(\RR^n)_{\GL(n,\RR)}) \cong H_k(\GL(n,\RR); \St_n(F)\otimes \Q)
\end{equation}
with $\St_n(F)$ the Steinberg module of $\GL(n,F)$.
Furthermore, the product and coproduct above endow $\shB_\RR$ with the structure of a differential graded commutative Hopf algebra.
\end{thm}

In the following, we will usually set $\RR=\Z$, and then omit the subscript,
\[
\shB := \shB_{\Z}.
\]

\subsection{The explicit map to the Steinberg module}
It will be important to know the (or one) explicit form of the isomorphism \eqref{equ:AMP}. To write it down we will use the augmented reduced normalized chain complex $\hat N\T_\bullet(V)$ of Section \ref{sec:red aug}.
Then there is an explicit map of dg vector spaces 
\begin{equation}\label{equ:chi}
\chi\colon \Sh(\RR^n) \longrightarrow \hat N \T_\bullet(F^n) [-2+\delta_{n,0}]
\end{equation}
from the sharbly complex into the normalized (rational) chain complex of the Tits building on $F^n$.
This map $\chi$ sends a non-zero sharbly $[v_1,\dots,v_k]\in \Sh(\RR^n)$ to
\begin{equation}\label{equ:chi def}
\chi([v_1,\dots,v_k]) = 
\begin{cases}
    \emptyset & \text{if $n=k=0$ or ($n=k=1$ and $v_1\neq 0$)} \\
    0 & \text{if $k\neq n$ or any $v_j=0$} \\
    \sum_{\sigma\in S_n}(-1)^\sigma 
C_\sigma &\text{if $k=n$ and all $v_j\neq 0$}
\end{cases},
\end{equation}
where $C_\sigma$ is the chain of subspaces
\[
\vspan_F(v_{\sigma(1)})\subset \vspan_F(v_{\sigma(1)},v_{\sigma(2)})\subset \cdots \subset\vspan_F(v_{\sigma(1)},\dots,v_{\sigma(n-1)}).
\]
Note that all subspaces are proper since the $v_j\neq 0$ and since we use only $n-1$ of the $n$ vectors in the chain. For a non-zero sharbly all subspace inclusions are strict since $v_1, \dots, v_n$ span $F^n$. Hence $C_\sigma$ is a normalized chain.
In the special case $n=1$ the sharbly $[v]$ is sent to the empty chain of subspaces $\emptyset$. In the special case $n=0$ the only non-trivial sharbly $[\,]$ is sent to the empty chain $\emptyset$ as well by convention. Note that if $v_1,\dots, v_n$ is an $F$-basis, $\chi([v_1,\dots,v_n])$ is the corresponding apartment.

The following lemma is well known and only included for completeness.
\begin{lemma}\label{lem:AMP chi}
    The morphism $\chi$ is a quasi-isomorphism of dg vector spaces for each $n\in \mathbb N_0$. It is furthermore $\GL(n,\RR)$-equivariant.
\end{lemma}
\begin{proof}
    We first check well-definedness, i.e., that $\chi$ respects the sharbly relations.
    The nullity relation \eqref{equ:sh null rel} is respected by the definition of $\chi$, and the relations \eqref{equ:sh perm rel} and \eqref{equ:sh sign rel} are obviously respected as well.
    It remains to check that \eqref{equ:sh span rel} is respected. 
    So suppose that the non-zero vectors $v_1,\dots,v_n$ do not span $F^n$.
    All degenerate terms are zero in the normalized complex. We claim that the non-degenerate terms in \eqref{equ:chi def} cancel in pairs. That is, to every $\sigma\in S_n$ such that $C_\sigma$ is non-degenerate there is a unique $\sigma'\in S_n$ with $\sigma'\neq \sigma$ such that $C_\sigma=C_{\sigma'}$, and that $\sigma'$ is of opposite parity.
    We may assume $\sigma=\mathrm{id}$ for notational simplicity.
    Then $v_1,\dots,v_{n-1}$ are linearly independent by the non-degeneracy of $C_\sigma$ and $v_{n}=\sum_{j=1}^{k}a_j v_j$ is a linear combination of the first say $k$ (for some $1\leq k\leq n-1$) vectors with $a_k\neq 0$. Then $\sigma'=(kn)$ (the transposition of $k$ and $n$) is the unique permutation $\sigma'$ such that $C_\sigma=C_{\sigma'}$. Hence $\chi([v_1,\dots,v_n])=0$ and $\chi$ is well-defined.
    %But if the non-zero vectors $v_1,\dots,v_n$ do not span $F^n$, then $\chi([v_1,\dots,v_n])$ is a linear combination of degenerate  degenerate and hence zero in the normalized complex.

    We next check that $\chi$ respects the differentials.
    Since $\chi(x)$ is only nonzero for sharblies $x\in \Sh(\RR^n)$ of length $n$, we have to check the compatibility equation
    \[
    \chi(dx) = d\chi(x)
    \]
    only for $x$ of length $n$ or $n+1$, and in either case one of the sides of the equation vanishes trivially.
    First, let $x=[v_1,\dots,v_n]$ be of length $n$, so that $dx=0$.
    Abbreviate $V_{\sigma,j}=\vspan(v_{\sigma(1)},\dots,v_{\sigma(j)})$.
    Then 
    \[d\chi(x) = \sum_{j=1}^{n-1} (-1)^{j-1} \sum_{\sigma\in S_n}
    (-1)^\sigma 
    (V_{\sigma,1} \subset \cdots \subset V_{\sigma,j-1}\subset \widehat{V_{\sigma,j}}
    \subset V_{\sigma,j+1}\subset \cdots \subset V_{\sigma,n-1}).\]
    We note that the terms on the right-hand side appear in pairs with opposite sign: If we interchange $\sigma(j)$ and $\sigma(j+1)$ the chain on the right-hand side stays the same, but the sign changes.
    Hence $d\chi(x)=0$ as desired.

    Next suppose that $x=[v_1,\dots,v_{n+1}]$ has length $n+1$.
    Then we can rewrite
    \[
    \chi(dx) = \pm 
    \sum_{\sigma\in S_{n+1}} (-1)^\sigma 
    (V_{\sigma,1} \subset \cdots \subset V_{\sigma,n-1}).
    \]
    Swapping $\sigma(n)$ and $\sigma(n+1)$ leaves the expression on the right-hand side the same while flipping the sign $(-1)^\sigma$, so again all terms cancel in pairs and $\chi(dx) =0$ as desired.

It remains to check that the map $\chi$ is a quasi-isomorphism.
For the degenerate cases $n=0,1$ this is done by an easy hand computation, so we focus on the case $n\geq 2$.
Then we already know that both sides of \eqref{equ:chi} have isomorphic homology, concentrated in degree $n$. 
The map $\chi$ then sends the cycle $[v_1,\dots,v_n]$ to the (sum of nondegenerate top dimensional simplices of) the apartment corresponding to the basis.
This is the same map as (the $\phi_{\Sh}$) used in \cite[Theorem 5]{AGMC}, and hence it is an isomorphism on homology by that theorem.
    %\todo[inline]{Find some better reference???}
    Finally, the $\GL(n,\RR)$-equivariance of $\chi$ is evident.
\end{proof}

\section{Lie algebraic preliminaries}
All Lie algebras, Hopf algebras and algebraic objects in this section will be understood over the ground field $\Q$.
\subsection{Maurer--Cartan elements and twisting}
\label{sec:twisting}
A differential graded Lie algebra is a differential graded vector space $(\fg,d)$ that in addition carries a graded Lie algebra structure with bracket $[-,-]$ such that $d[x,y]=[dx,y]+(-1)^{|x|}[x,dy]$ for all homogeneous $x,y\in \fg$.
For the purposes of this section, let us work in homological conventions so that $d$ reduces the homological degree by one.
Then a \emph{Maurer--Cartan element} is an element $\alpha\in \fg$ of degree $-1$ that satisfies the Maurer--Cartan equation 
\[
d\alpha +\frac 12 [\alpha,\alpha] =0.
\]
Given a Maurer--Cartan element $\alpha$, we may define the \emph{twisted} (by $\alpha$) dg Lie algebra $\fg^\alpha$ to be the same graded Lie algebra as $\fg$, but equipped with the twisted differential 
\[
d_\alpha := d + [\alpha,-].
\]
Note that the Maurer--Cartan equation guarantees that $d_\alpha^2=0$. 

Likewise, if $M$ is a $\fg$-module and $\alpha\in \fg$ is a Maurer--Cartan element, then we may also twist the module $M$ by $\alpha$, to obtain a $\fg^\alpha$-module $M^\alpha$.
Concretely, $M^\alpha$ is the same graded module as $M$, with the same $\fg$-action, but with differential 
\[
d_M^\alpha = d_M + \alpha\cdot,
\]
where $d_M$ is the original differential on $M$ and $\alpha\cdot$ is the action of $\alpha$ on $M$.
Again, the Maurer--Cartan equation guarantees that $(d_M^\alpha)^2=0$.

Generally, if our module $M$ happens to be equipped with some additional algebraic structure that is respected by the $\fg$-action, then the twisted module  $M^\alpha$ is still equipped with the same algebraic structure, which is respected by the $\fg^\alpha$-action.

For our purposes, this will happen in the following situation: Assume that $\fg=\fc^*$ is the dual of a dg Lie coalgebra $\fc$. Then in particular $\fg$ acts on $\fc$ compatibly with (i.e., by coderivations of) the dg Lie coalgebra structure on $\fc$. Concretely, for $x \in \fg$ and $\xi \in \fc$, the action is
\[ x \cdot \xi = -(x \otimes \id) \delta(\xi).\]
This action extends to the universal enveloping coalgebra $U^c\fc$ by
\[ x \cdot \xi = -(x \otimes \id) \Delta(\xi) + %(-1)^{|x|(|\xi|+1)}
(\id \otimes x)\Delta(\xi)\]
for homogeneous elements $x\in \fg$ and $\xi \in U^c\fc$.\footnote{Here we use the usual Koszul sign conventions, namely 
\[
(\id \otimes x) (\xi' \otimes \xi'')
:=
(-1)^{|x| |\xi'|}\xi' x( \xi'').
\]}
Hence, given a Maurer--Cartan element $\alpha\in \fg$, the twisted object $\fc^\alpha$ is still a dg Lie coalgebra. 
Likewise, the universal enveloping coalgebra $U^c\fc$ may be twisted by a Maurer--Cartan element $\alpha\in \fg$ and $(U^c\fc)^\alpha=U^c\fc^\alpha$ is still a dg Hopf algebra.

We also remark that the twisting construction is functorial. Concretely, assume that $\fg$ and $\fh$ are dg Lie algebras and $\phi\colon\fg\to \fh$ is a morphism of dg Lie algebras. Then if $\alpha\in \fg$ is a Maurer--Cartan element, $\phi(\alpha)$ is a Maurer--Cartan element in $\fh$, and the same linear map $\phi$ defines a morphism of dg Lie algebras
\[
\phi \colon \fg^\alpha\longrightarrow \fh^{\phi(\alpha)} 
\]
between the twisted objects.
Furthermore, let $M$ be a $\fg$-module, $N$ be an $\fh$-module and let $f\colon M\to N$ be module morphism, that is, a morphism of dg vector spaces such that $f(x\cdot m)=\phi(x)\cdot f(m)$ for all $x\in \fg$, $m\in M$.
Then the same map $f$ yields a morphism between the twisted modules 
\[
f\colon M^\alpha \longrightarrow N^{\phi(\alpha)}.
\]

\subsection{Twisting in Hopf algebras}
\label{sec:Hopf twisting}
Next let $A$ be a dg Hopf algebra. We can associate two dg Lie algebras to $A$. First, we say that an element $x\in A$ is primitive if its coproduct $\Delta x$ satisfies 
\[
\Delta x=x\otimes 1 + 1 \otimes x.
\]
The primitive elements $P(A)\subset A$ then form a dg Lie algebra.
Second, let $I=\ker\epsilon \subset A$ be the augmentation ideal, i.e., the kernel of the counit $\epsilon$. Then the indecomposables of $A$ are $Q(A):=I/I^2$.
They form a dg Lie coalgebra. Their dual space 
\begin{equation}\label{equ:fgA def}
\fg_A:=(Q(A))^*
\end{equation}
is hence a dg Lie algebra. Concretely, elements of $\fg_A$ are the same as linear maps $\lambda\colon A\to \Q$ such that $\lambda(1) = 0$ and 
\begin{equation}\label{equ:coprimitive}
\lambda(xy)=\lambda(x)\epsilon(y) + (-1)^{|x||\lambda|}\epsilon(x) \lambda(y)
\end{equation}
for all homogeneous $x,y\in A$.
Both the dg Lie algebra of primitives $P(A)$ and the dual indecomposables $\fg_A$ act on $A$ compatibly with the dg Hopf algebra structure, i.e., by biderivations. 

This means that given a Maurer--Cartan element in either dg Lie algebra we can twist the Hopf algebra $A$ as described in the previous subsection.

We focus on the case of a Maurer--Cartan element $\alpha\in \fg_A$, as this is the relevant case below.
The twisted Hopf algebra $A^\alpha$ is then the same underlying graded Hopf algebra $A$, but with the differential altered by adding the action of $\alpha$. This action is, on $x \in A$,
\begin{equation}\label{equ:lie on hopf}
\alpha\cdot x = -\sum (\alpha(x') x'' -(-1)^{|x'||x''|}\alpha(x'')x'),
\end{equation}
using Sweedler notation for the coproduct 
\[
\Delta\colon x\longmapsto \sum x' \otimes x''. 
\]

In the situation relevant to this paper, we will furthermore have a morphism of dg Hopf algebras $\Phi\colon B\to A$, and a Maurer--Cartan element $\alpha\in \fg_{A}$. Then via the dual morphism $\Phi^*\colon \fg_A\to \fg_B$ we obtain a Maurer--Cartan element $\Phi^*(\alpha)\in \fg_B$. Furthermore, by twisting we obtain a (the same) morphism of dg Hopf algebras
\[
\Phi\colon B^{\Phi^*(\alpha)} \longrightarrow A^{\alpha}.
\]

\begin{rem}
Under additional finiteness conditions the object $A^*$ is again a Hopf algebra and $\fg_A$ are the primitives of $A^*$. Since our chain models (like $\shB$ above) are often not of finite type, we prefer to directly work with $\fg_A$ rather than $A^*$.
\end{rem}

\subsection{Pre-Lie algebras and their universal enveloping algebras}
\label{sec:pre Lie}
A (dg) pre-Lie algebra is a dg vector space $\fg$ together with a morphism of dg vector spaces
\[
\circ \colon \fg\otimes \fg\longrightarrow \fg,
\]
such that the associator $(x\circ y) \circ z- x\circ(y\circ z)$ is graded symmetric in the last two arguments $y$ and $z$.
Let $(\fg,\circ)$ be a pre-Lie algebra. Then $\fg$ is also a Lie algebra with the Lie bracket 
\[
[x,y]= x\circ y - (-1)^{|x||y|}y\circ x.
\]
It is well known (see e.g. Oudom--Guin
\cite{OG}) that the universal enveloping algebra of the Lie algebra $(\fg, [-,-])$ has a very nice explicit form, using the underlying pre-Lie structure. We recall this here.
First, we need the operations (with $S^k(-)$ the $k$-fold (graded) symmetric product)
\[
\circ_{k} \colon \fg \otimes S^k\fg \longrightarrow \fg
\]
defined recursively such that $x\circ_1 y=x\circ y$ and 
\begin{multline*}
x\circ_k (y_1,\dots,y_k) = 
(x\circ_{k-1} (y_1,\dots,y_{k-1}) ) \circ y_k 
\\-
\sum_{j=1}^{k-1} (-1)^{|y_k|(|y_{j+1}|+\cdots + |y_{k-1}|)}
x\circ_{k-1} (y_1,\dots,y_j\circ y_k , \dots,y_{k-1}). 
\end{multline*}

%\begin{ex}
%    Our running example will be $(\fg,\circ)=(\fGCc_0,\bullet)$.
%    In this case 
%    \[
%\gamma \circ_k (\nu_1,\dots,\nu_k) 
%    \]
%    is obtained by simultaneously inserting the graphs $\nu_1,\dots,\nu_k$ into $k$ distinct vertices of $\gamma$.
%\end{ex}

It is convenient to drop the subscript $k$ and just write 
\[
x\circ (y_1,\dots,y_k) := x\circ_k (y_1,\dots,y_k).
\]
Next, consider the symmetric coalgebra $S\fg$. 
For $x_1,\dots,x_p\in \fg$ we denote by $x_1\cdots x_p\in S\fg$ the corresponding monomial. In other words, juxtaposition denotes the commutative product on the Hopf algebra $S\fg$.

Then, we define in $S\fg$ another (in general non-commutative) product $*$ such that 
\[
(x_1\cdots x_p) * (y_1\cdots y_q)
=
\sum_{f\colon\{1,\dots,q\} \to \{0,\dots,p\}}\pm 
y_{f^{-1}(0)} (x_1\circ y_{f^{-1}(1)}) \cdots (x_p\circ y_{f^{-1}(p)}),
\]
where we use the following notation for a finite set $I=\{i_1,\dots,i_j\}$:
\[
y_I := y_{i_1}\cdots y_{i_j} \in S\fg.
\]
By convention we declare that $y_\emptyset =1$ and $x_j\circ 1=x_j$.

\begin{prop}[{\cite[Theorem 2.12]{OG}}]
\label{prop:U pre Lie}
    Together with the cocommutative coproduct on $S\fg$ the operation $*$ defines a Hopf algebra structure on $S\fg$, and one has an isomorphism of Hopf algebras 
    \[
    U\fg \cong (S\fg,\Delta,*)
    \]
    with the universal enveloping algebra of $(\fg,[-,-])$. 
\end{prop}

%\begin{rem}
%    Let $\alpha\in \fg$ be a Maurer--Cartan element in the dg Lie algebra $\fg$. Then the twisted dg Lie algebra $\fg^\alpha$ is the same graded Lie algebra but with differential $d+[\alpha,-]$. We have that twisting commutes with taking the universal enveloping algebra, i.e., 
%    \[
%    U(\fg^\alpha) = (U\fg)^\alpha.
%    \]
%\end{rem}

\subsection{Pre-Lie coalgebras and their universal enveloping coalgebras}
\label{sec:pre Lie co}
Proposition \ref{prop:U pre Lie} readily dualizes to pre-Lie coalgebras. A (graded) pre-Lie coalgebra is a (graded) vector space $\fc$ equipped with a linear map of vector spaces
\[
\circ^* \colon \fc \longrightarrow \fc \otimes \fc,
\]
the pre-Lie coproduct, whose coassociator is graded symmetric in the last two positions. A pre-Lie coalgebra is also a (graded) Lie coalgebra with the Lie cobracket (the cocommutator) $\Delta_{[,]} = \circ^* - \tau \circ^*$, where $\tau\colon\fc\otimes \fc \to \fc\otimes \fc$ swaps the position of the factors, with the standard Koszul signs.
We say that a pre-Lie coalgebra is conilpotent if for any $x\in \fc$ there is an $n$ such that every $n$-th application of $\circ^*$ to $x$ is zero. In the following we shall always assume that $\fc$ is conilpotent.

Dually to before, we can associate to a pre-Lie coalgebra the operations 
\[
\circ_k^* \colon \fc \longrightarrow \fc \otimes S^k\fc.
\]
Note that we shall use as the target the symmetric coinvariants $S^k\fc= (\fc^{\otimes k})_{S_k}$, while strictly speaking the dual to $\circ_k$ of the last section would take values in the invariants $(\fc^{\otimes k})^{S_k}$ instead. However, since we are working in characteristic zero, both spaces are isomorphic, with the isomorphism being
\begin{align*}
(\fc^{\otimes k})_{S_k} &\longrightarrow (\fc^{\otimes k})^{S_k} \\
x_1\cdots x_k &\longmapsto \sum_{\sigma\in S_k} \pm 
x_{\sigma(1)} \otimes \cdots \otimes x_{\sigma(k)}
.
\end{align*}

Extend $\circ^*$ to an operation 
\begin{align*}
    \hat \circ^* \colon S^{k-1} \fc &\longrightarrow S^{k} \fc \\
    \hat \circ^* (x_1\cdots x_{k-1}) &= 
    \sum_{j=1}^{k-1} 
    x_1\cdots \mu(\circ^* x_j) \cdots x_{k-1},
\end{align*}
where $\mu\colon \fc \otimes \fc \to S^2\fc$ is the multiplication.
We then define $\circ_k^*$ recursively such that $\circ_1^*=\circ^*$ and
%\todo[inline]{Rewrite}
\begin{equation}\label{equ:circ star}
\circ_k^* = \frac 1k 
\left( (\id_\fc \otimes \mu)((\circ_{k-1}^*) \otimes \id_{\fc}) \circ^*
-
(\id_\fc \otimes \hat \circ^*)  \circ_{k-1}^* \right),
\end{equation}
%\begin{equation}\label{equ:circ star}
%\circ_k^* = \frac 1k 
%\left( (\id_\fc \otimes \mu)((\circ_{k-1}^*) \otimes \id_{\fc}) \circ^*
%-
%\sum_{j=1}^{k-1} 
%(\id_{\fc^{\otimes j}}\otimes \circ^* \otimes \id_{\fc^{\otimes k-1-j}})  \circ_{k-1}^* \right),
%\end{equation}
with
\[
\mu \colon S^{k-1} \fc \otimes \fc \longrightarrow S^{k} \fc  
\]
again the commutative product in the symmetric algebra.
Then we define a coassociative coproduct 
\[\Delta\colon S\fc\longrightarrow S\fc\otimes S\fc\]
on $S\fc$ by the formula 
\begin{equation}\label{equ:Delta c def}
\Delta (x_1\cdots x_n) =  
\left(x_1\otimes 1 + 1\otimes x_1 + \sum_{k_1=1}^\infty \circ_{k_1}^* x_1\right) \cdots \left(x_n\otimes 1 + 1\otimes x_n + \sum_{k_n=1}^\infty \circ_{k_n}^* x_n\right).
\end{equation}
Note that by conilpotency the infinite sums are in fact finite on every argument.
Dual to Proposition \ref{prop:U pre Lie} we have the following result.
\begin{prop}[{dual of \cite[Theorem 2.12]{OG}}]
\label{prop:U pre co Lie}
    Let $\fc$ be a conilpotent graded pre-Lie coalgebra.
    Together with the commutative product $\cdot$ on $S\fc$ the operation $\Delta$ defines a graded commutative Hopf algebra structure on $S\fc$, and one has an isomorphism of Hopf algebras 
    \[
    U^c\fc \cong (S\fc,\cdot,\Delta)
    \]
    with the universal enveloping coalgebra of the graded Lie coalgebra $(\fc,\Delta_{[,]})$. 
\end{prop}

\section{Graph complexes}

In this section, we introduce several varieties of graph complexes. A summary of all of them is found in Section \ref{sec:gc notation} below.

\label{sec:graphs}
\subsection{Definition and recollections}
Let $\fG_0$ be the graded vector space spanned by pairs $(\gamma, o)$ consisting of a finite combinatorial graph $\gamma$ together with an orientation datum $o$, which is an ordering of the set of edges. The grading $|(\gamma,o)|$ is given by the number of edges in $\gamma$. The graphs can be connected or not, vertices may have any positive valence\footnote{That is, we disallow the graph consisting of a single vertex without edge. This could in fact also be included, but will yield slightly more cumbersome formulas.}, and we allow \emph{tadpoles}, i.e. edges with identical endpoints, and multiple edges between the same pair of vertices. The empty graph is included, but the graph consisting of a single vertex without edges is not.
Such a pair is considered an element of homological degree $k$, if $\gamma$ has $k$ edges. 
For the orientation datum, we use the notation $o=e_1\wedge \cdots \wedge e_k$ to denote the ordering of the $k$ edges $e_1<e_2<\cdots < e_k$. 
We quotient by two relations:

\begin{itemize}
    \item (Isomorphism) Let $\phi\colon\gamma\to \nu$ be a graph isomorphism. Let $o$ be an orientation
    for $\gamma$ and $\phi_*(o)$ the natural orientation of $\nu$ induced by $\phi$. Then we set 
    \begin{equation}\label{equ:gc iso rel}
    (\gamma,o) = (\nu,\phi_*(o)).
    \end{equation}
    \item (Orientation change) Suppose that $o'=e_{\sigma(1)}\wedge \cdots \wedge e_{\sigma(n)}$ is another orientation of $\gamma$ obtained by applying some permutation $\sigma\in S_n$. Then we set 
    \begin{equation}\label{equ:gc perm rel}
    (\gamma,o')= (-1)^\sigma (\gamma,o).
    \end{equation}
\end{itemize}
Note that the isomorphism relation makes graphs with multiple edges between the same endpoints vanish.

We will often suppress the orientation datum $o$ in the notation and just talk about the graphs $\gamma$, with the implicit understanding that an orientation datum is fixed.
We let 
\[
\fGc_0 \subset \fG_0
\]
be the graded subspace spanned by connected graphs.
Furthermore, we define 
\begin{align*}
    \fGC_0 &= \fG_0^* & \fGCc_0 &= \fGc_0^*
\end{align*}
to be the respective graded dual spaces.

The graded vector space $\fGc_0$ is a graded pre-Lie coalgebra with the pre-Lie coproduct given by subgraph contraction
\[
\circ^* (\gamma,e_1\wedge\cdots \wedge e_k)
=
\sum_{\gamma'\subset \gamma} \pm (\gamma/\gamma',
\bigwedge_{k\atop e_k\notin \gamma'} e_k)
\otimes (\gamma', \bigwedge_{k\atop e_k\in \gamma'} e_k),
\]
where the sum is over connected proper subgraphs $\gamma'$ of $\gamma$ with at least one edge, $\gamma/\gamma'$ is obtained by contracting $\gamma'$ inside $\gamma$ to one vertex, and the sign is the sign of the unshuffle permutation moving the edges of $\gamma'$ to the right of the other edges.
We can also rewrite the same formula as 
\begin{equation}\label{equ:gc coprod conn}
\circ^* (\gamma,e_1\wedge\cdots \wedge e_k)
=
\sum_{\substack{ E\subsetneq E\gamma \\ E\neq \emptyset \\ \gamma(E)\text{ conn.} }} \pm (\gamma/E,
\bigwedge_{e\notin E} e)
\otimes (\gamma(E), \bigwedge_{e\in E} e),
\end{equation}
where the sum is now over subsets of edges, such that the subgraph $\gamma(E)$ consisting of these edges and the adjacent vertices is connected.
Note that the pre-Lie coalgebra $\fGc_0$ is obviously conilpotent.

The graded pre-Lie coalgebra structure makes the dual space $\fGCc_0$ into a graded pre-Lie algebra. Because $\fGc_0$ has a canonical (up to sign) basis given by graph isomorphism classes, we identify elements in $\fGCc_0$ with oriented graphs as well. Since we are using homological grading, in $\fGCc_0$, the number of edges of $\gamma$ is $-|(\gamma,o)|$. Omitting orientations, we denote the pre-Lie product by
\[
\gamma\bullet \nu.
\]
It is obtained by summing over all ways of inserting $\nu$ in place of a vertex of $\gamma$.
More precisely, we have $\gamma\bullet \nu =\sum_{v\in V\gamma}\gamma \bullet_v \nu$, and $\gamma \bullet_v \nu$ is obtained by deleting vertex $v$ from $\gamma$ and then summing over all ways of reconnecting the half-edges previously attached to $v$ to vertices of $\nu$.
Note that this operation is graded with respect to the number of edges. 
Furthermore, the resulting graphs are endowed with the orientation datum that is obtained by just concatenating those of $\gamma$ and $\nu$.

The special element 
\begin{equation}
    \label{equ:L def}
L=
\begin{tikzpicture}
    \node[int] (v) at (0,0) {};
    \node[int] (w) at (0.7,0) {};
    \draw (v) edge (w);
\end{tikzpicture}
\end{equation}
is a Maurer--Cartan element in $(\fGCc_0, d=0)$, i.e., $[L,L]=0$.
The Lie bracket with $L$ hence defines a differential on $\fGCc_0$, and by duality also one on $\fGc_0$.
Similarly, the element 
\begin{equation}
    \label{equ:T def}
T=\tadpole
\end{equation}
is also a Maurer--Cartan element in $\fGCc_0$, as well as the linear combination $L+T$, i.e., $[T,T]=[L,T]=0$. 
We consider $\fGCc_0$ as a dg Lie algebra with differential $d=\frac12 [L,-]$.

Next, we define the dg Lie subalgebras $\GC_0\subset \GC_0^+ \subset \fGCc_0$, where $\GC_0$ consists of linear combinations of graphs all of whose vertices are at least trivalent, and $\GC_0^+=\GC_0\oplus \Q T$ is just the one-dimensional extension by the element $T$.

Dually, we have quotients (and dg Lie coalgebra maps) $\fGc_0\twoheadrightarrow \G_0^+\twoheadrightarrow \G_0$, where to obtain $\G_0$ we declare all graphs with univalent or bivalent vertices to be zero, and for $\G_0^+$ we just do not set $T$ to zero.
We then have 
\begin{align*}
    \G_0^* &= \GC_0 &  (\G_0^+)^* &= \GC_0^+.
\end{align*}

Finally, we note that since our homological degree in $\fGc_0$ is just the number of edges, the complex is positively graded and finite dimensional in each homological degree. Of course, these statements hold a fortiori for the quotients $\G_0^+$ and $\G_0$. 

\begin{rem}\label{rem:gc pairing}
Note that isomorphism classes of graphs may be used to index both a basis of $\fGc_0$, as well as a basis of the dual space $\fGCc_0$.
We use the convention that the pairing $\fGCc_0\otimes \fGc_0\to \Q$ sends a pair of graphs with orientation data $(\gamma,o),(\gamma',o')$ to 0 if $\gamma$ and $\gamma'$ are not isomorphic, and if they are isomorphic to $\pm |\operatorname{Aut}(\gamma)|$, with the sign of the permutation that maps the orientation data to each other. This convention ensures the natural formulas for the dg Lie structure on $\fGCc_0$ mentioned above.
\end{rem}

In the following we will say that an edge in a graph is a \emph{bridge} if removing this edge while keeping all vertices increases the number of connected components of the graph. Equivalently, a bridge is an edge that is not part of any cycle. 
An \emph{antenna} is an edge adjacent to a univalent vertex. Antennas are in particular bridges.

\subsection{The universal enveloping coalgebra of graphs}
Next, note that obviously the space of possibly disconnected graphs $\fG_0$ can be understood as the symmetric algebra on the connected part $\fGc_0$, 
\[
\fG_0 \cong S \fGc_0,
\]
with the commutative product being the ($\Q$-linear extension of the) disjoint union of graphs, and the unit being the empty graph. 
Then Proposition \ref{prop:U pre co Lie} endows $\fG_0$ with a coproduct $\Delta$ such that we have an isomorphism of graded Hopf algebras
\[
U^c \fGc_0 \cong (S\fGc_0, \cdot, \Delta) \cong
(\fG_0, \cup, \Delta).
\]
Since the differential is given by twisting with a Maurer--Cartan element (namely $L/2$ above), this immediately implies that the above isomorphism is also an isomorphism of dg Hopf algebras.
It remains to deduce the explicit combinatorial formulas for the operation $\circ_k^*$ and the coproduct $\Delta$. 

\begin{lemma}
    Let $(\gamma,e_1\wedge\cdots \wedge e_\ell)\in \fGc_0$, then
\begin{equation}\label{equ:circ gc}
\circ_k^* 
(\gamma,e_1\wedge\cdots \wedge e_\ell)
=
\sum_{\emptyset \neq E\subsetneq E(\gamma) \atop \# \pi_0 \gamma(E) = k } \pm (\gamma/E,
\bigwedge_{e\notin E} e)
\otimes (\gamma(E_1), \bigwedge_{e\in E_1} e)\cdots (\gamma(E_k), \bigwedge_{e\in E_k} e), 
\end{equation}
where the sum is over proper subsets of edges such that the induced graph $\gamma(E)$ has exactly $k$ connected components $\gamma(E_1),\dots, \gamma(E_k)$ corresponding to the partition $E=E_1\sqcup \cdots \sqcup E_k$. The sign is the sign of the permutation bringing the edges in the indicated order.
\end{lemma}

\begin{proof}
    We trace the construction of Section \ref{sec:pre Lie co} for the case $\fc = \fGc_0$.
    To see the above formula, we argue by induction. For $k=1$ this is just \eqref{equ:gc coprod conn}. 
For $k\geq 2$ we use the induction hypothesis in the definition \eqref{equ:circ star}. We obtain 
\begin{multline*}
(\id\otimes \mu)(\circ_{k-1}^*\otimes \id)\circ^*
(\gamma,e_1\wedge\cdots \wedge e_\ell)
=
\\
\sum_{\emptyset \neq E\subsetneq E(\gamma) \atop \gamma(E)\text{ conn.}}
\sum_{\emptyset \neq E'\subsetneq E(\gamma)\setminus E \atop \# \pi_0 (\gamma/E)(E') = k-1}
\pm (\gamma/(E\cup E'), \bigwedge_{e\notin E\cup E'} e)
\otimes \\\otimes 
((\gamma/E)(E_1'),\bigwedge_{e\in E_1'} e)
\cdots ((\gamma/E)(E_{k-1}'),\bigwedge_{e\in E_{k-1}'} e)
(\gamma(E), \bigwedge_{e\in E} e).
\end{multline*}
We can decompose the sum into two types of terms. Note that $\gamma/E$ has a distinguished vertex $*$ at which $\gamma(E)$ was contracted. The first type happens when $*$ is not part of any $\gamma(E_j')$, then $(\gamma/E)(E_j')=\gamma(E_j')$. The sum over all $E$ and $E'$ where $E$ is not connected to any components in $E'$ exactly produces $k$ times \eqref{equ:circ gc}, with the prefactor accounting for one of the $k$ subgraphs being singled out. The second type happens, when $*$ is part of a $\gamma(E'_j)$.

The other term in the recursive definition of $\circ_k^*$ is given by
\begin{multline*}    
(\id_\fc \otimes \hat \circ^*)  \circ_{k-1}^*(\gamma,e_1\wedge\cdots \wedge e_\ell) = \\
\sum_{j=1}^{k-1}\sum_{\emptyset \neq E''\subsetneq E(\gamma) \atop \# \pi_0 \gamma(E'') = k-1 } \pm (\gamma/E'',
\bigwedge_{e\notin E''} e)
\otimes (\gamma(E''_1), \bigwedge_{e\in E''_1} e)\cdots  \\
\left(\sum_{\emptyset \neq E\subsetneq E''_j \atop \gamma(E)\text{ conn.}} \pm (\gamma(E''_j)/E,\bigwedge_{e \in E''_j \setminus E} e)(\gamma(E), \bigwedge_{e\in E}e) \right) \cdots(\gamma(E''_{k-1}), \bigwedge_{e\in E''_{k-1}} e).
\end{multline*}
We observe that setting $E'= E''\setminus E$ and in particular $E'_j = E''_j\setminus E$ and $E'_i = E''_i$ for $i\neq j$ identifies this sum with the sum of the terms of the second type in the first equation.

Therefore, the terms of the second type cancel in the recursive definition \eqref{equ:circ star} of $\circ_k^*$. This shows the lemma.
\end{proof}

\begin{prop}
    Let $\Delta$ be the coproduct on $\fG_0$ obtained from the coproduct $\Delta$ on $U^c \fGc_0\cong S\fGc_0$ as in Proposition \ref{prop:U pre co Lie} via the isomorphism $\fG_0\cong S\fGc_0$. Then for $(\Gamma,e_1\wedge\cdots \wedge e_\ell) \in \fG_0$ we have
    \begin{equation}\label{equ:gc coprod}
    \Delta (\Gamma,e_1\wedge\cdots \wedge e_\ell)
=
\sum_{E\subset E(\Gamma)} \pm (\Gamma/E,
\bigwedge_{e\notin E} e)
\otimes (\Gamma(E), \bigwedge_{e\in E} e),
\end{equation}
where the sum is over all subsets of edges, $\Gamma(E)$ is the subgraph consisting of these edges and the adjacent vertices, $\Gamma/E$ is obtained by contracting the edges in $E$, and the sign $\pm$ is the sign of the unshuffle permutation moving the edges in $E$ to the right of the other edges $\{e_1,\dots,e_\ell\}\setminus E$. If $E=\emptyset$, then we interpret $(\Gamma(E), \bigwedge_{e\in E} e)$ as the empty graph.
If the contraction $\Gamma/E$ produces vertices of valence zero, then these vertices are removed from the graph.
This applies in particular to the case $E=E(\Gamma)$, when $(\Gamma/E,
\bigwedge_{e\notin E} e)$ is understood as the empty graph.
\end{prop}

\begin{proof}
Finally, inserting \eqref{equ:circ gc} into \eqref{equ:Delta c def}, we obtain precisely \eqref{equ:gc coprod}. Note here that the terms $1\otimes x_j$ appearing in \eqref{equ:Delta c def} are exactly produced by our convention that if all edges in a connected component are contracted, then the resulting zero-valent vertex is removed. The terms $x_j\otimes 1$ correspond naturally to the case that no edge of the $j$-th connected component appears in $E$.
\end{proof}

%The constructions of Section \ref{sec:pre Lie} immediately dualize to pre-Lie coalgebras $\fc$. 
%Let us specialize to the case $\fc=\fGc_0$.
%We will say that an edge in a graph is an \emph{antenna} if it connects to a univalent vertex.
%\begin{prop}
 %   The universal enveloping coalgebra of $\fGc_0^L$, the dg Lie coalgebra twisted by $L$ defined in \eqref{equ:L def}, then has the following structure:
 %   \begin{itemize}
%        \item The underlying graded vector space is $\fG_0=S(\fGc_0)$, i.e., the space of possibly non-connected graphs. 
 %       \item The commutative product is the product on the symmetric algebra. Combinatorially, this is the disjoint union of graphs. The empty graph is the unit of this product.
  %      \item The reduced coproduct 
   %     satisfies the formula 
    %    \begin{equation}\label{equ:gc coprod}
    %    \Delta' \Gamma = \sum_{E\subsetneq E(\Gamma)\atop E \neq \emptyset} \pm \Gamma/E \otimes \Gamma(E), 
     %   \end{equation}
       % where the sum runs over proper subsets $E$ of edges of $\Gamma$, $\Gamma/E$ is the graph obtained by contracting all edges in $E$, and $\Gamma(E)$ is the subgraph of $\Gamma$ consisting of the edges of $E$ and all vertices incident to these edges.%\
       % \item 
    
    The differential on $U^c \fGc_0\cong \fG_0$ is obtained by twisting with the Maurer--Cartan element $\frac 12 L$, i.e., it has the form, see also \eqref{equ:lie on hopf}, 
        \begin{equation}\label{equ:L twist}
        d\colon \Gamma \longmapsto -\frac 12\sum L(\Gamma') \cdot \Gamma'' + \frac 12 \sum (-1)^{|E(\Gamma')|} \Gamma' \cdot L(\Gamma''),
        \end{equation}
        where we use Sweedler notation $\Delta\Gamma = \sum \Gamma' \otimes \Gamma''$ to denote the coproduct. 
    Using the explicit form \eqref{equ:gc coprod} of the coproduct on $\fG_0$ we may show the following combinatorial formula for the differential.
    \begin{lemma}\label{lem:d formula}
        Let $(\Gamma,e_1\wedge\cdots \wedge e_k)\in \fG_0$.
        Then we have that 
        \[
        d (\Gamma,e_1\wedge\cdots \wedge e_k) = 
        d_c (\Gamma,e_1\wedge\cdots \wedge e_k)
        + 
        d_b (\Gamma,e_1\wedge\cdots \wedge e_k),
        \]
        where 
        \begin{align} \label{equ:gc c diff}
            d_c(\Gamma,e_1\wedge\cdots \wedge e_k) &=
            \sum_{j=1\atop {e_j \text{ non-antenna,}\atop \text{non-tadpole}}}^k (-1)^{j-1} (\Gamma/e_j,e_1\wedge\cdots\wedge \hat e_j \wedge \cdots \wedge e_k) 
            \\
            \label{equ:gc b diff}
            d_b(\Gamma,e_1\wedge\cdots \wedge e_k) &=
            \sum_{j=1\atop {e_j \text{ bridge,} \atop \text{non-antenna}}}^k (-1)^{j} (\Gamma-\{e_j\},e_1\wedge\cdots\wedge \hat e_j \wedge \cdots \wedge e_k) .
        \end{align}
    \end{lemma}
    \begin{proof}
        Note that $L(\Gamma)=2$ if $\Gamma= \begin{tikzpicture}
    \node[int] (v) at (0,0) {};
    \node[int] (w) at (0.7,0) {};
    \draw (v) edge (w);
\end{tikzpicture} \in \fG_0$ and zero for all other graphs, see Remark \ref{rem:gc pairing}.
    Consider simultaneously the definition of the differential \eqref{equ:L twist} and the explicit formula for the coproduct \eqref{equ:gc coprod}.
    For the first term in \eqref{equ:L twist} to be nonzero, we need that $\Gamma'$ is (a multiple of) $L$. This means that in \eqref{equ:gc coprod} we need to pick up a term in which $E$ are all edges except for one, say edge $e_j$, and $e_j$ needs to be a bridge, possibly an antenna. In this case $\Gamma''$ is obtained by deleting $e_j$ from $\Gamma$. The sign is $(-1)^{j}$.

    Next consider the second term of \eqref{equ:L twist}.
    For this to be non-zero we need to have $E=\{e_j\}$, consisting of a single element, in \eqref{equ:gc coprod}.
    Here $e_j$ can be any edge that is not a tadpole. The graph $\Gamma/E$ is then obtained from $\Gamma$ by contracting $e_j$, and the sign is $(-1)^{k-1+(k-j)}=(-1)^{j-1}$. 

    Note that for $e_j$ an antenna, deleting $e_j$ and contracting $e_j$ produces the same graph. Hence these terms cancel between the first and second sum in \eqref{equ:L twist}.
    The remaining terms of the first sum are exactly \eqref{equ:gc b diff}, and those from the second yield \eqref{equ:gc c diff}.
    \end{proof}
  %  \end{itemize}
%\end{prop}

For later reference we also record the following:
\begin{lemma}\label{lem:T action}
Let $\nabla\colon \fG_0 \to \fG_0$ be the action of $\frac12 T$ (see \eqref{equ:T def}). Then for $(\Gamma,e_1\wedge\cdots \wedge e_k)\in \fG_0$ we have that 
\begin{equation}\label{equ:fG0 nabla}
\nabla (\Gamma,e_1\wedge\cdots\wedge e_k)
=\sum_{\substack{j=1\\\text{$e_j$ no tadpole, nor bridge}}}^k (-1)^j (\Gamma-e_j,e_1\wedge\cdots \wedge\hat e_j\wedge \cdots \wedge e_k),
\end{equation}
\end{lemma}
\begin{proof}
Using the coproduct, the defining formula is 
\begin{equation*}%\label{equ:T twist}
\nabla\colon \Gamma \longmapsto -\frac 12\sum T(\Gamma') \cdot \Gamma'' + \frac 12 \sum (-1)^{|E(\Gamma')|} \Gamma' \cdot T(\Gamma'').
\end{equation*}
The first term can only be nonzero if $\Gamma'=T$. This is true only if in \eqref{equ:gc coprod} $E=\{e_1,\dots,\hat e_j,\dots,e_k\}$ with $e_j$ a non-bridge. But $e_j$ can be a tadpole. The sign is $(-1)^j$.
On the other hand $\Gamma''$ can only be a tadpole if $E=\{e_j\}$ and $e_j$ is a tadpole. The sign is $(-1)^{j-1}$. Hence the tadpole terms cancel and only the terms in \eqref{equ:fG0 nabla} remain.
\end{proof}

\begin{rem}
Note that the graph $\Gamma/E$ in \eqref{equ:gc coprod} is obtained by individually contracting all edges in $E$. This is not the same as contracting the whole subgraph $\Gamma(E)$ to a point, but rather the same as contracting each of its connected components separately. In particular, the sum of the loop orders of the graphs appearing on the right-hand side of \eqref{equ:gc coprod} is the same as the loop order of $\Gamma$, so that the coproduct respects the loop order grading.
%For the last point we note that non-triviality of the connected components means that each component has to have at least one edge.
%
%Note that the Hopf algebra structure on $\fG_0$ is the associated graded Hopf algebra defined in \cite[Corollary 5.20]{BCGP}.  
%The full one can similarly be obtained by twisting with $T$. 
\end{rem}

\subsection{Wheel classes}
The $n$-wheel graphs for $n\geq 3$ are the graphs of the form
\[
W_n=
\begin{tikzpicture}
    \node[int](v) at(0,0) {};
    \node (v1) at (0:1) {$\cdots$};
    \node[int] (v2) at (72:1) {};
    \node[int] (v3) at (144:1) {};
    \node[int] (v4) at (216:1) {};
    \node[int] (v5) at (288:1) {};
    \draw (v) edge (v1) edge (v2) edge (v3) edge (v4) edge (v5)
    (v1) edge (v2) (v2) edge (v3) (v3) edge (v4) (v4) edge (v5) (v5) edge (v1);
\end{tikzpicture}
\]
with a total of $2n$ edges and $n+1$ vertices.
Fixing an arbitrary orientation (i.e., an ordering of the set of edges), we may consider $W_n$ as an element of $\G_0$. As such, $W_n=0$ for even $n$ by symmetry, while $W_n\neq 0$ for odd $n$.
Furthermore, the wheel classes are obviously closed, $dW_n=0$, since any edge contraction results in a double edge. It is known that for odd $n$ the wheel classes define non-trivial homology classes $W_n\in H_{2n}(\G_0)$, see e.g. Ward \cite[Corollary 1.2]{Ward21} for an explicit statement.

\subsection{Deformed differential $\nabla$}
Since the tadpole $T\in \GC_0^+$ of \eqref{equ:T def} is a Maurer--Cartan element,
we may twist the dg Lie algebra $(\GC_0^+,d)$ by $\frac 12T$.
The resulting Lie algebra has the same underlying vector space and Lie bracket, just the differential is changed from $d$ (vertex splitting) to $d+\frac12[T,-]$. For easier notation we will generally write $\nabla$ for the action of $\frac12T$, so that the twisted differential is $d+\nabla$.
We will not use this fact later, but note that combinatorially, $\nabla$ acts on a graph $\gamma\in \GC_0^+$ by summing over all ways of adding one additional edge to $\gamma$, between any pair of distinct vertices. This is the dual statement to Lemma \ref{lem:T action}.

While the cohomology of $(\GC_0,d)$ is largely unknown, the cohomology of the twisted complex $(\GC_0,d+\nabla)$ can be fully computed:

\begin{thm}[{Khoroshkin--Willwacher--\v Zivkovi\'c \cite{KWZ}, Willwacher \cite[Theorem 4.3]{WillwacherTriconnected}}]
\label{thm:nabla cohom}
We have 
\[
H(\GC_0^+,d+\nabla) \cong \Q T \oplus \bigoplus_{k\geq 1} \Q \sigma_{2k+1}, 
\]
where the classes $\sigma_{2k+1}$ are in cohomological degree $4k+2$.
The class $\sigma_{2k+1}$ pairs non-trivially with $W_{2k+1}\in \G_0$.
\end{thm}
Note however that the $\sigma_{2k+1}$ are represented by a longer linear combination of graphs, not just the wheel graph.

Next, we may also twist the dual dg Lie coalgebra $\G_0^+$ by the action of $T/2$, which we still denote by $\nabla$.
Explicitly, $\nabla$ acts on a graph $(\gamma,e_1\wedge\cdots\wedge e_k)\in \G_0^+$ by 
\[
\nabla (\gamma,e_1\wedge\cdots\wedge e_k)
=\sum_{\substack{j=1\\\text{$e_j$ no tadpole}}}^k (-1)^j (\gamma-e_j,e_1\wedge\cdots \wedge\hat e_j\wedge \cdots \wedge e_k),
\]
where $\gamma-e_j$ is the graph $\gamma$ with the edge $e_j$ removed, and we set the term to zero if $\gamma-e_j$ is disconnected or has vertices of valence $\leq 2$. 
For example, we have that 
\[
\nabla W_n=0,
\]
since removing any edge in the wheel graph produces a bivalent vertex.
Hence the wheel classes are $(d+\nabla)$-cycles, $(d+\nabla)W_n=0$.
By Theorem \ref{thm:nabla cohom} they pair non-trivially with cocycles in the dual complex and are hence non-trivial in homology $H(\G_0,d+\nabla)$, as long as $n$ is odd.
The dual statement to Theorem \ref{thm:nabla cohom} then reads:
\begin{cor}
We have an isomorphism of graded Lie coalgebras
\[
H(\G_0^+,d+\nabla) \cong \Q t \oplus \bigoplus_{k\geq 1} \Q W_{2k+1}, 
\]
where $t$ is represented by the tadpole graph \eqref{equ:T def}, and the Lie coalgebra structure on the right-hand side is trivial.
\end{cor}
\begin{proof}
    Given Theorem \ref{thm:nabla cohom} it just remains to check that the Lie cobracket on the right-hand side vanishes.
    This is clear for $t$. For the $W_{2k+1}$ note that any proper subgraph of $W_{2k+1}$ has vertices of valence $\leq 2$, so the cobracket is trivial already on chains.
\end{proof}
We also record the analogous statement for the universal enveloping coalgebra.
\begin{cor}\label{cor:KWZ}
We have an isomorphism of Hopf algebras
\[
H(U^c\G_0^+,d+\nabla) \cong \Q[ t, W_3, W_5,\dots], 
\]
where $\Q[\cdots]$ denotes the free graded commutative algebra in the given generators.
\end{cor}

We may think of $U^c\G_0^+$ as the quotient of $\fG_0$ obtained by setting to zero all graphs with vertices of valence $\leq 2$, except possibly bivalent vertices on tadpole connected components.
Then the explicit combinatorial formula for $d$ is that of Lemma \ref{lem:d formula} and that for $\nabla$ is \eqref{equ:fG0 nabla}.
The product in $U^c\G_0^+$ then becomes the disjoint union of graphs. For example, $W_3W_5\in U^c\G_0^+$ can be interpreted as the non-connected graph that is the union of $W_3$ and $W_5$.

\subsection{Table of graph complex notation}
\label{sec:gc notation} $ $

\vspace{1em}

\begin{tabular}{c|p{10cm}}
    Symbol & Description \\
    \hline 
    $\fG_0$ & The full graph complex, spanned by oriented, possibly disconnected graphs with vertices of valence $\geq 1$. The empty graph is allowed.
    \\
    $\fGC_0$ & The dual space of $\fG_0$. \\
    $\fGc_0$ & The subspace of $\fG_0$ spanned by non-empty connected graphs.
    \\ 
    $\fGCc_0$ & The dual space of $\fGc_0$. \\
    $\G_0$ & The subspace of $\fGc_0$ spanned by graphs whose vertices have valence $\geq 3$. \\
    $\GC_0$ & The dual space of $\G_0$. \\
    $\G_0^+$ & Obtained from $\G_0$ by adding the tadpole graph $\tadpole$. \\
    $\GC_0^+$ & The dual space of $\G_0^+$. 
\end{tabular}

\section{The map $\Phi$}\label{sec:Phi}
For a graph $\Gamma$ we denote by $C^1(\Gamma;\Z)$ the space of integral cochains of $\Gamma$, i.e. linear combinations of (directed) edges. 
Let 
\[
C^1(\Gamma;\Z) \longrightarrow H^1(\Gamma;\Z) \cong \Z^{\beta_1(\Gamma)}
\]
be the natural quotient map. Now let $(\Gamma,e_1\wedge \cdots \wedge e_k)\in \fG_0$ be a graph with an orientation, i.e. an ordering of edges.
Then, we define a sharbly in $\Sh(H^1(\Gamma;\Z))$ as 
\[
\Phi(\Gamma,e_1\wedge \cdots \wedge e_k) = 
[e_1,\dots,e_k],
\]
where we understand $e_j$ as an element of $C^1(\Gamma;\Z)$ by assigning an arbitrary direction to $e_j$, and we understand it as an element of $H^1(\Gamma;\Z)$ via the map $C^1(\Gamma;\Z) \to H^1(\Gamma;\Z)$.
Passing further to $\GL(H^1(\Gamma;\Z))$-coinvariants then yields a linear map 
\begin{gather*}
    \Phi\colon \fG_0 \longrightarrow \shB \\
    \Gamma \longmapsto [e_1,\dots,e_k].
\end{gather*}

\begin{prop}\label{prop:Phi Hopf map}
    The map 
    \begin{gather*}
        \fG_0 \longrightarrow \shB \\
        (\Gamma,e_1\wedge\cdots\wedge e_k) \longmapsto [e_1,\dots,e_k]
    \end{gather*}
    is a morphism of dg Hopf algebras for $\RR=\Z$.
\end{prop}

\begin{proof}
First, we verify that the map is well-defined. To this end, first note that our arbitrary choice of edge directions does not change the sharbly by the relation \eqref{equ:sh sign rel}.
Second, we clearly have 
\[
\Phi(\Gamma,e_{\sigma(1)}\wedge\cdots\wedge e_{\sigma(k)})
=
(-1)^\sigma
\Phi(\Gamma,e_{1}\wedge\cdots\wedge e_{k})
\]
 by \eqref{equ:sh perm rel}, so that the graph complex relation \eqref{equ:gc perm rel} is respected.
Finally, let $\phi\colon\Gamma\to \Gamma'$ be an isomorphism and $\phi^*\colon H^1(\Gamma')\to H^1(\Gamma)$ be the induced map on degree 1 cohomology.
Then, we have that in $\shB$, due to taking $\GL$-coinvariants
\begin{align*}
    \Phi(\Gamma',\phi(e_1)\wedge\cdots\wedge \phi(e_{k})) 
    &= 
    [\phi(e_1),\dots, \phi(e_{k})]
    =
    [\phi^* \phi(e_1),\dots, \phi^* \phi(e_{k})] 
    \\&=
    [e_1,\dots,e_k] 
    =
    \Phi(\Gamma,e_{1}\wedge\cdots\wedge e_{k}).
\end{align*}
Hence, the map $\Phi$ is well defined.

Compatibility with the product is obvious.

For compatibility with the coproduct fix a graph $(\Gamma,e_1\wedge\cdots\wedge e_k)\in \fG_0$. Fix a subset $I=\{i_1,\dots, i_m\}$ of the edge indices with $1\leq i_1<\cdots <i_m\leq k$ and let $J=\{1,\dots,k\}\setminus I=\{j_1,\dots,j_{k-m}\}$ be its complement. Let $E=\{e_{j_1},\dots,e_{j_{k-m}}\}\subset E(\Gamma)$ be the corresponding set of edges of $\Gamma$. Then, we compare the two coproduct formulas \eqref{equ:gc coprod} and \eqref{equ:sh coprod}, and we claim that the summand corresponding to $E$ in \eqref{equ:gc coprod} matches the corresponding summand for $I$ in \eqref{equ:sh coprod} via $\Phi$. 
First, note that there is a short exact sequence
\begin{align}\label{equ:h1 exact}
    0\longrightarrow H^1(\Gamma/E;\Z) \longrightarrow H^1(\Gamma;\Z) \longrightarrow H^1(\Gamma(E);\Z) \longrightarrow 0 .
\end{align}
Now the $E$-term of \eqref{equ:gc coprod} is mapped via $\Phi\otimes \Phi$ to
\begin{equation}\label{equ:PhiPhi1}
\pm
\Phi(\Gamma/E,o') \otimes \Phi(\Gamma(E),o'') 
=
\pm 
\underbrace{[e_{i_1}, \dots , e_{i_{m}}]}_{\in \Sh(H^1(\Gamma/E;\Z))}
\otimes 
\underbrace{[e_{j_1}, \dots , e_{j_{k-m}}]}_{\in \Sh(H^1(\Gamma(E);\Z))},
\end{equation}
where the sign is the sign of the corresponding unshuffle permutation.
Next, we compare this to the $I$-summand in the coproduct \eqref{equ:sh coprod} of $\Phi(\Gamma,e_1\wedge\cdots\wedge e_k)$, namely
\begin{equation}\label{equ:PhiPhi2}
\pm 
\underbrace{[e_{i_1}, \dots , e_{i_{m}}]}_{\in\Sh(\vspan(e_{i_1}, \dots,e_{i_{m}})_{\sat})}
\otimes 
\underbrace{[\bar e_{j_1}, \dots , \bar e_{j_{k-m}}]}_{\in \Sh(H^1(\Gamma;\Z)/ \vspan(e_{i_1}, \dots,e_{i_{m}})_{\sat} )}.
\end{equation}
The sign is again the sign of the unshuffle permutation and hence agrees with that of \eqref{equ:PhiPhi1}.
The first factor of \eqref{equ:PhiPhi2} takes values in $\Sh(\vspan(e_{i_1}, \dots,e_{i_{m}})_{\sat})$, with $\vspan(e_{i_1}, \dots,e_{i_{m}})_{\sat}\subset H^1(\Gamma;\Z)$. 
Note that this is the same as $H^1(\Gamma/E;\Z)$.
Similarly, the second factor of \eqref{equ:PhiPhi2} takes values in $\Sh(H^1(\Gamma;\Z)/ \vspan(e_{i_1}, \dots,e_{i_{m}})_{\sat} )$, and we have $H^1(\Gamma;\Z)/ \vspan(e_{i_1}, \dots,e_{i_{m}})_{\sat}\cong H^1(\Gamma(E);\Z)$.
Hence \eqref{equ:PhiPhi1} and \eqref{equ:PhiPhi2} agree, as desired.

Finally, we check that the differential is respected by the map $\Phi$.
To this end, recall the defining formulas \eqref{equ:gc c diff}, \eqref{equ:gc b diff} and \eqref{equ:sh diff}.
We first have 
\begin{align*}
&\Phi((d_c+d_b)(\Gamma,e_1\wedge\cdots\wedge e_k))\\
&=
\sum_{\substack{j=1\\ e_j\text{ non-antenna,}\\ \text{non-tadpole}}}^k (-1)^{j-1}
[e_1,\dots,\hat e_j,\dots, e_{k}]
-\sum_{\substack{j=1\\ e_j\text{ bridge,} \\ \text{non-antenna}}}^k (-1)^{j-1}
[e_1,\dots,\hat e_j,\dots, e_{k}]
\\
&=
\sum_{\substack{j=1\\ e_j\text{ non-bridge,}\\ \text{non-tadpole}}}^k
(-1)^{j-1}
[e_1,\dots,\hat e_j,\dots, e_{k}].
\end{align*}
Here, we used that deleting a bridge edge and contracting a bridge edge has precisely the same effect after applying $\Phi$.
Compare this expression to 
\[
d\Phi(\Gamma,e_1\wedge\cdots\wedge e_k)
=
\sum_{j=1\atop e_j\neq 0}^k (-1)^{j-1}
[e_1,\dots,\hat e_j,\dots, e_{k}].
\]
Note that the image of an edge $e$ in cohomology is zero if and only if it is not part of any cycle in the graph, and this is true if and only if it is a bridge. Finally, the tadpoles in the graph correspond to coloops in its image under $\Phi$, whose deletion produces the zero sharbly. Hence both formulas above agree and we are done. 
\end{proof}

\subsection{Proof of Theorem \ref{thm:main 1}}
Given Proposition \ref{prop:Phi Hopf map}, we will finish the proof of Theorem \ref{thm:main 1} by showing the following two lemmas.

\begin{lemma}
    The map $\Phi$ of Proposition \ref{prop:Phi Hopf map} factors through the universal enveloping coalgebra of the quotient $\G_0^+ \twoheadleftarrow \fGc_0$.
    \[  
    \begin{tikzcd}
        U^c \fGc_0 \ar{r}{\Phi} \ar{d} & \shB \\
        U^c \G_0^+ \ar{ur}
    \end{tikzcd}.
    \]
\end{lemma}
We will use the letter $\Phi$ to denote the map $\Phi\colon U^c \G_0^+\to \shB$ as well.
\begin{proof}
    We have to check that $\Phi$ vanishes on all graphs that have univalent or bivalent vertices, except for the case that the bivalent vertices are those of connected components of the form $t=\tadpole$. 

    First assume the graph $\Gamma$ has a univalent vertex $v$. Then the unique edge attached to $v$ always maps to zero under the map $C^1(\Gamma;\Z)\to H^1(\Gamma;\Z)$ since it is the coboundary of $v$. Hence the sharbly $\Phi(\Gamma)$ has a zero column, and thus vanishes by \eqref{equ:sh null rel}. 
    
    Next assume that $\Gamma$ has a bivalent vertex $v$.
    If its two incident half-edges are part of the same edge, then the connected component of $v$ is just $\tadpole$.
    Otherwise, the bivalent vertex is connected to two different edges.
    \[
\begin{tikzpicture}
    \node[int,label=90:{$v$}] (v) at (0,0) {};
    \draw (v) edge node[above]{$e$} +(-1,-.5) edge node[above]{$e'$} +(1,-.5) ;
\end{tikzpicture}
    \]
    The cohomology classes in $H^1(\Gamma;\Z)$ defined by $e,e'$ agree up to sign.
    Hence the sharbly $\Phi(\Gamma,o)$ has two columns that are identical up to sign. Thus $\Phi(\Gamma,o)=0$ by the sign and permutation relations \eqref{equ:sh perm rel}, \eqref{equ:sh sign rel}.
\end{proof}

The following lemma was proven with Brown--Schnetz \cite[Theorem 4(i)]{BrownSchnetz} as its main ingredient.

\begin{lemma}[{Ash--Miller--Patzt \cite[Lemma 4.1]{AMP}}]\label{lem:wheelnonzero}
    For $n\geq 3$ odd let $W_{n}$ be the wheel cycle with $n$ spokes and hence $2n$ edges.
    Then $\Phi(W_n)\in \shB$ represents a non-trivial homology class. 
\end{lemma}
We provide an independent proof of this result in Remark \ref{rem:wheel indep} below.

\section{Deformed differentials}

%\subsection{Deformed differential on the graph complex}
%The tadpole element $T\in \GC_0^+$ of \eqref{equ:T def} is a Maurer--Cartan element. It follows that one can twist the differential on $\GC_0^+$ or $\G_0^+$ by $T$, i.e., we add the action $\nabla$ by $T$ to the differential
%\[
%d\to d+\nabla.
%\]
%Then we still have $(d+\nabla)^2=0$, and the differential $d+\nabla$ is still a (co)derivation with respect to the Lie (co)algebra structure.
%Precisely, on $\GC_0^+$ the action is just the bracket with $T$, $\nabla=[T,-]$, while on $\G_0^+$ the action is by taking the cobracket and then contracting the first slot with $T$, i.e., in Sweedler notation:
%\[
%\nabla x = \sum T(x') x''-(-1)^{|x'||x''|}T(x'')x'.
%\]
%The operation $\Delta$ on the Lie coalgebra then induces an operation $\nabla:U^c\G_0^+\to U^c\G_0^+$ on the universal enveloping coalgebra.
%Concretely, if we think of $U^c\G_0^+$ as a complex of possibly non-connected graphs, then $\nabla$ sums over all ways of removing one non-bridge edge. In particular, the application of $\nabla$ reduces the first Betti  number of the graph by exactly one.

\subsection{Deformed differential on sharblies and first part of Theorem \ref{thm:main 2}}\label{sec:deformed D}
We consider the dg Lie algebras dual to indecomposables $\fg_{\shB}$ and $\GC_0^+=\fg_{U^c\G_0^+}$, see \eqref{equ:fgA def}.
Our Hopf algebra map $\Phi\colon U^c\G_0^+\to \shB$ induces a dual map of dg Lie algebras
\[
\Phi^* \!\colon \fg_{\shB} \longrightarrow \GC_0^+.
\]
We now define a Maurer--Cartan element $\eta\in \fg_{\shB}$, as the linear map $\eta\colon \shB\to\Q$ that projects any $x\in \shB$ to the sum of the coefficients of the non-zero sharblies of the form $[a]\in \Sh(\Z)_{\GL(1,\Z)}$ in $x$.

\begin{lemma}\label{lem:eta}
    The element $\eta \in\shB^*$ satisfies the following properties:
    \begin{enumerate}[(i)]
    \item It is in $\fg_\shB$, i.e., \eqref{equ:coprimitive} holds.
    \item It satisfies $d\eta=[\eta,\eta]=0$. In particular, it is a Maurer--Cartan element in the dg Lie algebra $\fg_\shB$.
    \item We have $\Phi^*(\eta)=\frac12 T\in \GC_0^+$.
    \end{enumerate}
\end{lemma}
\begin{proof}
    The first assertion is true for degree reasons. 
    For the second, the Maurer--Cartan equation reads 
    \[
    \eta d + (\eta \otimes \eta)\Delta = 0.
    \]
    This in fact follows from $\eta d=0$ and $(\eta \otimes \eta) \Delta = 0$ separately vanishing.
    We check
    \[
    \eta d [a,b]= \eta ([b]-[a])=0,
    \]
    using that we must have $a,b\neq 0$ by \eqref{equ:sh null rel}.
    Similarly, let $v_1,v_2\in \Z^2$ be two vectors that span $\Q^2$. Then
    % We may assume $v_1=\begin{pmatrix}a\\0\end{pmatrix}$, $v_2=\begin{pmatrix}b\\c\end{pmatrix}$ with $a,c\neq 0$. Then we have, with $g:=gcd(b,c)$
    \[
    (\eta \otimes \eta) \Delta [v_1,v_2]
    =
    (\eta \otimes \eta)([v_1]\otimes [\bar v_2]-[v_2]\otimes [\bar v_1])=-1+1=0.
    \]
    For the last assertion, note that $\eta$ is supported in rank 1, but the only graph $\Gamma\in \GC_0^+$ of rank 1 is $\Gamma=T$.
    It is clear that $\eta(\Phi(t))=1$, where $t\in \G_0^+$ is the tadpole graph. (I.e., $t$ is equal to $T$, but interpreted as an element of the dual graph complex). Hence $\Phi^*(\eta)=\frac 12T$, see Remark \ref{rem:gc pairing}. 
\end{proof}

Given the Maurer--Cartan element $\eta$, we may twist the differential on $\shB$ by its action, see Section \ref{sec:Hopf twisting}.
That is, we consider $\shB$ with the differential $d+D$ with 
\[
Dx=-\sum (\eta(x')x''- (-1)^{|x'||x''|}\eta(x'')x'),
\]
using again Sweedler notation for the coproduct.
Let us explicitly state the action of $D$ on some non-zero sharbly $x=[v_1,\dots,v_N]\in \Sh(\Z^n)$.
\begin{equation}\label{equ:D formula}
\begin{aligned}
D [v_1,\dots,v_N]
&= 
-\sum_{j=1}^N
(-1)^{j-1}
[v_1,\dots,\hat v_j,\dots,v_N]_{\Z^n/(\Z v_j)_{\sat}}
\\&\quad \quad+
\sum_{\substack{ j=1 \\ \vspan_\Q(v_1,\dots,\hat v_j,\dots,v_N)\neq \Q^n } }^N
(-1)^{j-1}
[v_1,\dots,\hat v_j,\dots,v_N]_{\vspan(v_1,\dots,\hat v_j,\dots,v_N)_{\sat}}
\\
&= D_1 x + D_2x.
\end{aligned}
\end{equation}
Here we denoted the terms of the first sum by $D_1 x$ and those of the second sum by $D_2x$. 

Since $\Phi^*\eta=\frac12 T$, we conclude that automatically $\Phi$ is a morphism of dg Hopf algebras with the deformed differentials
\[
\Phi \colon (U^c\G_0^+,d+\nabla) \longrightarrow (\shB, d+D),
\]
see again Section \ref{sec:Hopf twisting}.
This already shows the first part of Theorem \ref{thm:main 2}.

\begin{rem}
    We remark that the construction of the deformed differential $d+D$ on $\shB$ does not rely on the fact that $\RR=\Z$ and would work for any PID $\RR$ in the same way, and yields a Hopf algebra $(\shB_{\RR},d+D)$ and a morphism of dg Hopf algebras $(U^c\G_0^+,d+\nabla) \to (\shB_\RR, d+D)$. 
\end{rem}
\begin{rem}\label{rem:coloop cancellation}
    For $\RR=F$ a field, in particular $\RR=\Q$, the formula for $D$ slightly simplifies in that coloop terms cancel between $D_1$ and $D_2$ after taking $\GL(n,\RR)$-coinvariants. Namely, suppose that $v_j$ is a coloop in \eqref{equ:D formula}. Then we may always assume that in suitable coordinates $v_j=(0,\dots,0,1)$ and all other $v_i\in F^{n-1}\subset F^{n}$. So it is clear that (up to ambient isomorphism)
    $[v_1,\dots,\hat v_j,\dots,v_N]_{F^n/(F v_j)}$ and $[v_1,\dots,\hat v_j,\dots,v_N]_{\vspan(v_1,\dots,\hat v_j,\dots,v_N)}$ are the same sharbly and hence cancel due to the opposite signs in \eqref{equ:D formula}. The resulting simplified formula for $D$ on $\shB_F$ is then 
    \begin{equation}\label{equ:D simplified}
    D [v_1,\dots,v_N]
    = 
    -
    \sum_{j=1\atop v_j\text{ no coloop}}^N
    (-1)^{j-1}
    [v_1,\dots,\hat v_j,\dots,v_N]_{F^n/F v_j}.
    \end{equation}
    
    The analogous cancelation of the coloop terms in $D_1$ and $D_2$ does not happen for integral sharblies. For example,
    \begin{align*}
    D_1 [(2,0), (1,2)] &=-[2] + [4] 
    &
    D_2 [(2,0), (1,2)] &=[1] - [2],
    \end{align*}
    while all vectors involved are coloops.
\end{rem}

\subsection{The Gale duality morphism}
\label{sec:gale}
Let $V$ be a $\Q$-vector space of rank $n=\vdim V$ and let $x=(v_1,\dots,v_k)$  
be a general tuple of vectors $v_j\in V$.
We may consider $x$ as (the matrix of) a linear map $\phi_x\colon \Q^k\to V$. Denote by $\phi_x^* \colon V^*\to (\Q^k)^*$ the dual map.
Then we define the Gale-dual $\Q$-vector space 
\[
V^G_x := \coker \phi_x^* = (\Q^k)^* / \im \phi_x^* \cong (\ker \phi_x)^*.
\]
Furthermore, let $e_j^*$ ($j=1,\dots,k$) be the standard basis vectors of $(\Q^k)^*$ and let $w_j$ be the image of $e_j^*$ in $V^G_x$.
Then we define the Gale-dual tuple to $x$ to be 
\[
G(x) := (w_1,\dots,w_k).
\]
The thus defined Gale duality $(x,V) \mapsto (G(x), V_x^G)$ has the following properties.

\begin{lemma}\label{lem:G properties}
    Let $V$ be a $\Q$-vector space of finite dimension $n$, and let $x=(v_1,\dots,v_k)$ be a tuple of vectors in $V$ that span $V$.
    \begin{enumerate}
        \item Let $\lambda_1,\dots,\lambda_k\in \Q^\times$ and $x'=(\lambda_1v_1,\dots,\lambda_kv_k)$. Then $V_x^G\cong V_{x'}^G$ and $G(x') = (\lambda_1^{-1}w_1, \dots,\lambda_k^{-1}w_k)$ if $G(x) = (w_1,\dots, w_k)$.
        \item For any permutation $\sigma\in S_k$ let 
        $x'=(v_{\sigma(1)},\dots,v_{\sigma(k)})$. 
        Then the action of $S_k$ on $(\Q^k)^*$ induces an isomorphism $\sigma\colon V_x^G \to V_{x'}^G$ and we have that $G(x')=\sigma G(x)$.
        \item For $A\in \GL(V)$ let $x'=Ax$. Then $V_{x'}^G=V_x^G$ and $G(x')=G(x)$.
        \item Let $G(x)=(w_1,\dots,w_k)$. Then $w_1,\dots,w_k$ span $V_x^G$. 
        \item %Assume that the $v_1,\dots,v_k$ span $V$.
        We have $(V_x^G)^G_{G(x)}\cong V$, with a natural isomorphism induced by $x$. 
        Under this isomorphism $G(G(x))=x$.
        \item $v_j$ is a coloop of $x$ (resp. $v_j=0$) if and only if $w_j=0$ (resp. $w_j$ is a coloop of $G(x)$).
    \end{enumerate}
\end{lemma}
\begin{proof}
    The first four assertions are obvious.
    For (5) we have that $(V_x^G)^G_{G(x)}$ is the dual of the kernel of the morphism 
    \[
    (\Q^k)^* \longrightarrow V_x^G = (\Q^k)^* / \im \phi_x^*.
    \]
    The kernel is $\im \phi_x^*$ and its dual is the coimage of $\phi_x$, i.e. $\Q^k/\ker \phi_x$. This is clearly identified with $\im \phi_x=V$, using the assumption that the $v_j$ span $V$, i.e., that $\phi_x$ is surjective. Concretely, the obvious identification (which depends on $x$) sends an element in $\Q^k/\ker \phi_x$ to its image under $\phi_x$. This means that the standard basis vectors are sent precisely to $v_1,\dots,v_k$, showing the claim. We remark that the same argument shows that if the $v_j$ do not span, then still $G^2(x)=x$, but the ambient space $V$ is replaced by $(V_x^G)_{G(x)}^G=\vspan(v_1,\dots,v_k)$.

    Finally, let us consider (6). We have that 
    $w_j=0$ iff $e_j^*\in \im \phi_x^*$, and this holds iff there exists a $\lambda\in V^*$ such that $\lambda(v_j)=1$ and $\lambda(v_i)=0$ for $i\neq j$.
    This is turn is true iff $v_j\notin \vspan\{v_i \mid i\neq j\}$, i.e., $v_j$ is a coloop.
    By applying $G$ and using that $G^2(x)=x$ (possibly in a smaller ambient space, but this does not matter) we see that $w_j$ is a coloop iff $v_j=0$. 
\end{proof}

Next, we extend the operation $G$ to rational sharblies
\begin{equation}\label{equ:Gale map}
G\colon \shB_\Q \longrightarrow \shB_\Q
\end{equation}
denoted by the same letter. To this end, let $x=[v_1,\dots,v_k]\in \Sh(V)$. We may assume that the $v_j\in V$ are all nonzero and together span $V$, otherwise we will set $G(x):=0$.
Then, we define $G(x)=[w_1,\dots,w_k]\in \Sh(V_x^G)$ as above.
Passing to $\GL$-coinvariants we obtain a morphism of graded vector spaces $G\colon\shB_\Q\to \shB_\Q$. 
Indeed, to see that this is well-defined note that the sharbly relation \eqref{equ:sh null rel} and \eqref{equ:sh span rel} are respected since we defined $G(x)=0$ on such null sharblies. The relation \eqref{equ:sh perm rel} and $\GL$-equivariance are respected by the assertions established in Lemma \ref{lem:G properties}. The sign relation \eqref{equ:sh sign rel} is respected due to the first assertion of Lemma \ref{lem:G properties}. 
We just note that since $G(x)=0$ if $x$ has a coloop, the morphism $G^2$ will not be the identity, but rather the projection to the subspace spanned by sharblies without coloops.

Finally, we check that $G$ is compatible with the differential in the following sense.

\begin{lemma}
For any $x\in \shB_\Q$ we have
\begin{align}\label{equ:G d D compat}
    G(dx) &= -D G(x), & G(Dx) &= -dG(x).
\end{align}
\end{lemma}
\begin{proof}
    Let $x=[v_1,\dots,v_k]\in \Sh(\Q^n)$ be a rank $n$ sharbly. We assume it is non-zero, in particular all vectors $v_j$ are non-zero and together span $\Q^n$ and no two vectors are the same up to sign.
    We first assume that $x$ has a coloop $v_j$.
    Then $G(x)=0$. Furthermore, all terms of $dx$ and $Dx$ also have $v_j$ (or $\bar v_j$) as a coloop, see \eqref{equ:d simplified} and \eqref{equ:D simplified}, so $G(dx)=G(Dx)=0$.
    Hence the above equations \eqref{equ:G d D compat} hold.

    So we may assume that $x$ has no coloop. 
    Then $dx=\sum_{j=1}^k(-1)^{j-1}[v_1,\dots,\hat v_j,\dots,v_k]$.
    Set $[w_1,\dots,w_k]:=G(x)$ and $x'=[v_1,\dots,\hat v_j,\dots,v_k]$. 
    Here $x'$ still spans $V=\Q^n$ since there is no coloop in $x$, and obviously all its vectors are non-zero.
    Note that, in the notation above, $\phi_{x'}^*\colon V^*\to (\Q^{k-1})^*$ factorizes as $V^*\xrightarrow{\phi_x^*}(\Q^{k})^*\xrightarrow{-/e_j^*} (\Q^{k-1})^*$.
    Hence it is clear that $V_{x'}^G=V_x^G/\Q w_j$ and
    \[
    G([v_1,\dots,\hat v_j,\dots,v_k])
    = 
    [\bar w_1,\dots, \widehat{\bar w_j},\dots, \bar w_k ]_{V_x^G/\Q w_j}.
    \]
    Comparing to the definition of $D$ we hence have $G(dx)=-DG(x)$.

    Next we consider (for our sharbly $x$ without coloops as above) 
    \[-Dx=\sum_{j=1}^k (-1)^{j-1} [\bar v_1,\dots,\widehat{\bar v_j},\dots \bar v_k]_{V/\Q v_j}.\]
    Redefine $x':=[\bar v_1,\dots,\widehat{\bar v_j},\dots \bar v_k]_{V/\Q v_j}$.
    However, it might happen that the sharbly $x'$ contains a zero vector, if $v_j$ is collinear to some other $v_i$. 
    Now the maps $\phi_x$ and $\phi_{x'}$ fit into the commutative diagram 
    \[
    \begin{tikzcd}
        0\ar{d}\ar{r} & K \ar{d} \ar{r}{\cong} & K' \ar{d} 
        \\
        \Q e_j \ar{r}\ar{d} & \Q^k \ar{r}\ar{d}{\phi_x} & \Q^{k-1} \ar{d}{\phi_{x'}}
        \\
        \Q v_j \ar{r} & V \ar{r} & V/\Q v_j
    \end{tikzcd},
    \]
    where $K, K'$ are the vertical kernels.
    All columns are exact and the bottom two rows, hence so is the upper row and $K\cong K'$. 
    Dualizing we have that $V_{x'}^G = K^*\cong (K')^* = V_x^G$.
    The vectors $w_i$ of the Gale dual sharbly $G(x)$ are the images of the basis vectors $e_i^*$, and this holds in the same manner on either side of equality 
    \[
    V_{x'}^G = (\Q^{k-1})^*/\im \phi_{x'}^* = (\Q^{k})^*/\im \phi_{x}^* = V_x^G, 
    \]
    so that 
    \[
    G(x') =
    \begin{cases}
        [w_1,\dots,\hat w_j, \dots, w_k] & \text{if $\bar v_i\neq 0$ for all $i\neq j$} \\
        0 & \text{otherwise}
    \end{cases}.
    \]
    It follows that 
    \[
    G(Dx)+dG(x)= \sum_{\substack{j=1 \\ j \text{ special}} }^k
    (-1)^{j-1} [w_1,\dots,\hat w_j, \dots, w_k],
    \]
    where $j$ is special if there are $v_i$ collinear to $v_j$ in $x$. Let us group the $v_j$ into equivalence classes with respect to collinearity. Assume by reordering that $v_1,\dots,v_r$ form one such equivalence class, with $r\geq 2$. It then suffices to check that 
    \begin{equation}\label{equ:collin cancel}
    \sum_{j=1}^r
    (-1)^{j-1} [w_1,\dots,\hat w_j, \dots, w_k]=0.
    \end{equation}  
    Note that the vectors $w_1,\dots,\hat w_j,\dots, w_r$ are coloops by the last assertion of Lemma \ref{lem:G properties} above. 
    We then distinguish two cases:
    \begin{itemize}
        \item If $r=2$, then $[w_2,w_3, \dots, w_k]=[\lambda w_1,w_3, \dots, w_k]$ in $\shB_\Q$ by the first property of Lemma \ref{lem:G properties}. 
        But since coloops can always be rescaled by the $\GL$-action, $[\lambda w_1,w_3, \dots, w_k]=[w_1,w_3, \dots, w_k]$ in $\shB_\Q$ and the sum \eqref{equ:collin cancel} vanishes.
        \item For $r\geq 3$ each sharbly appearing in \eqref{equ:collin cancel} contains at least two coloops.
        But such sharblies are zero in $\shB_\Q$. Namely, we may assume that in a suitable basis such a sharbly has the form 
        \[[(1,0,\dots,0), (0,1,0,\dots, 0), *]\]
        with $*$ representing further vectors whose first two coordinates vanish. Then 
        \begin{align*}
        [(1,0,\dots,0), (0,1,0,\dots, 0), *]
        &=
        -
        [(0,1,0,\dots, 0), (1,0,\dots,0), *]
        \\&=
        -
        [(1,0,\dots,0), (0,1,0,\dots, 0), *]=0,
        \end{align*}
        where we have used anti-symmetry for the first equation and the $\GL$-action for the second.
    \end{itemize}
    Hence we have shown that $G(Dx)+dG(x)=0$ as desired.
\end{proof}

\begin{ex}\label{ex:Gale dual graph}
    Suppose that $\Gamma\in U^c\G_0^+$ is a bridgeless graph with $k$ edges and let $x=\Phi(\Gamma)\in \Sh(H^1(\Gamma;\Z))$ be its sharbly representation. 
    We shall quietly identify $x$ with its image $x\in\Sh(H^1(\Gamma;\Q))$.
    The sharbly $x$ is essentially the loop-edge incidence matrix. More precisely, 
    \[
    \phi_x \colon \Q^k\cong C^1(\Gamma;\Q)\longrightarrow H^1(\Gamma;\Q)=:V
    \]
    is the natural quotient map from the $1$-cochains of $\Gamma$ into the first cohomology. There is the exact sequence
    \[
    0\longrightarrow H^0(\Gamma;\Q)\stackrel{\iota}{\longrightarrow } C^0(\Gamma;\Q)\stackrel{\partial^*}{\longrightarrow }  C^1(\Gamma;\Q) 
    \stackrel{\phi_x}{\longrightarrow} H^1(\Gamma;\Q)\longrightarrow 0,
    \]
    with $\partial^*$ the coboundary map.
    Thus $\ker \phi_x= \im\partial^* \cong \coker \iota$.
    Hence $V_x^G=\coker\phi_x^*=(\ker \phi_x)^* \cong \ker \iota^*$, with $\iota^*\colon C_0(\Gamma;\Q)\to H_0(\Gamma;\Q)$.
    %We hence have $\phi_x^*\colon H_1(\Gamma;\Q)\to C_1(\Gamma;\Q)$, and $V_x^G=\coker\phi_x^*=C_1(\Gamma;\Q)/H_1(\Gamma;\Q)$. The boundary map then yields an isomorphism $V_x^G\cong \ker \iota\subset C_0(\Gamma;\Q)$, with $\iota\colon C_0(\Gamma;\Q)\to H_0(\Gamma;\Q)$ the map sending a vertex to its connected component.
    %With this identification of $V_x^G\cong \ker \iota$, 
    The Gale dual sharbly $G(x)$ then encodes the morphism (edge-vertex incidence matrix)
    \[
    \phi_{G(x)} \colon \Q^k\cong C_1(\Gamma;\Q) \stackrel{\partial}{\longrightarrow} \ker\iota^* \subset C_0(\Gamma;\Q).
    \]
    Here some care has to be taken, in that for $x$ a sharbly containing a zero vector we have ``forcefully'' defined that $G(x)=0$, and $x=\Phi(\Gamma)$ contains a zero vector iff $\Gamma$ has a bridge.  
    
    It is furthermore interesting to note that the sharblies $\Phi(W_n)$ corresponding to the wheel graphs $W_n$ (for odd $n$) are anti-self-dual under $G$, $G(\Phi(W_n))=-\Phi(W_n)$.
    To see this, a natural basis of $H_1(W_n)$ is given by the $n$ triangles in the wheel. Let $e_1,\dots,e_n$ be the dual basis of $H^1(W_n)$.
    If we denote the edges in the rim of the wheel by $a_1,\dots, a_n$ and the spokes by $b_1,\dots, b_n$, then 
    \begin{equation}\label{equ:Phi Wn}
    \Phi(W_n) = [e_1,\dots,e_n,e_1-e_2,e_2-e_3,\dots,e_n-e_1]
    \end{equation}
    and 
    \begin{align*}
    \phi_x(a_i) &= e_i & \phi_x(b_i) &= e_i-e_{i+1},
    \end{align*}
    where indices are understood modulo $n$, i.e., $n+1\equiv 1$.
    Denoting the corresponding dual bases by $e_1^*,\dots, e_n^*$, respectively $a_1^*,\dots, a_n^*, b_1^*,\dots, b_n^*$ we hence have 
    \[
    \phi_x^*(e_i^*) = a_i^* + b_i^*- b_{i-1}^*. 
    \]
    Thus 
    \[
    V_x^G = \coker \phi_x^* = \Q^{2n} / \langle a_i^* + b_i^*- b_{i-1}^* \mid i=1,\dots, n\rangle.
    \]
    A basis of the latter space is given by $f_i=[b_i^*]$, $i=1,\dots, n$. The relations yield $[a_i^*]=f_{i-1}-f_{i}$. Hence 
    \begin{align*}
    G(\Phi(W_n)) &=
    [f_n-f_1,f_1-f_2,\dots ,f_{n-1}-f_n,f_1,\dots,f_n]
    \\&=
    -[f_1,\dots,f_n, f_n-f_1,f_1-f_2,\dots f_{n-1}-f_n]
    \\&=
    -[f_1,\dots,f_n,f_1-f_2,f_2-f_3,\dots,f_n-f_1],
    \end{align*}
    which is $-\Phi(W_n)$ in $\shB_\Q$.  
    
    If we identify $V_x^G\cong \ker \iota^*$ then $f_i$ represents the difference of the $i$-th vertex on the rim of the wheel and the center vertex, and $G(\Phi(W_n))$ can be understood as the (projected) edge-vertex incidence matrix.
\end{ex}

\begin{rem}
    The construction of the Gale duality map can be extended to $\shB_F$ for $F$ any field of characteristic zero. No special property of $\Q$ was used above.
\end{rem}

\section{The Goncharov cocycle}
For the purposes of this section, a (complex) cocycle is a linear map
\[
\alpha \colon (\shB,d+D) \longrightarrow \C
\]
such that 
\[
\alpha\circ (d+D) = 0.
\]
We will not necessarily require that $\alpha$ is of homogeneous degree. Rather $\alpha=\sum_j\alpha_j$, with the pieces $\alpha_j$ of degree $-j$, i.e., linear maps supported on the degree $j$ subspace $\shB_j$, that are individually cocycles.
%We will construct two cocycles in this section, the Goncharov cocycle $\alpha_{\mathrm G}$ and (a sketch of) the Brown--Borel cocycle $\alpha_B$.

\subsection{Goncharov's Grassmannian polylogarithm }
The \emph{Goncharov Grassmannian polylogarithm} $\GonL_{n}$ for $n\geq 2$ is a function that associates to a collection $(H_1,\dots, H_{2n})$ of $2n$ hyperplanes in $\CP^{n-1}$ a complex number $\GonL_n(H_1,\dots, H_{2n})$, introduced by Goncharov \cite{Goncharov}.
It is convenient to define hyperplanes from their coordinate vectors. So if $v\in \C^n$ is a non-zero complex $n$-vector, we define the corresponding hyperplane as 
\[
v^+ = \{z=(z_1,\dots,z_n) \mid v\cdot z = 0\} \subset \CP^{n-1},
\]
where we used homogeneous coordinates $(z_1,\dots,z_n)$ to denote a point in $\CP^{n-1}$.

To define $\GonL_n$, we will use Goncharov's integral (see \cite[Proposition 5.3]{Goncharov})
\begin{equation}\label{equ:psi G def}
\psi_n(v_1,\dots,v_{2n})
=
\int_{\CP^{n-1}} 
\log\left|\frac{v_2\cdot z}{v_1\cdot z}\right|
\bigwedge_{\alpha=3}^{2n} d\log\left|\frac{v_\alpha\cdot z}{v_1\cdot z}\right|
\end{equation}
for $z=(z_1,\dots,z_n)$ again homogeneous coordinates on $\CP^{n-1}$, and non-zero vectors $v_1,\dots,v_{2n}\in \C^{n}$. Then the Goncharov Grassmannian polylogarithm may be defined by, see \cite[Proposition 5.3]{Goncharov},\footnote{This is not the primary definition of $\GonL_n$ given in \cite{Goncharov}, but a property. It is the definition that is beneficial to use in our Section \ref{sec:RWvsG}, though.}
\begin{equation}\label{equ:psi is L}
\GonL_n(v_1^+,\dots, v_{2n}^+)
=
-(-4)^{n-1}(2\pi i)^{1-n} {\binom{2n-2}{n-1}}^{-1}
%(-4)^{n}(2\pi i)^{1-n}(2n)^{1-2n} {2n-2\choose n-1}^{-1}
\psi_n(v_1,\dots,v_{2n}).
\end{equation}
We will need the following main property of $\GonL_n$.
\begin{thm}[{Goncharov \cite[Theorem 4.3]{Goncharov}, \cite{GoncharovPersonal}}]
\label{thm:Goncharov} $ $
\begin{enumerate}
    \item The function $\GonL_n$ is skew-symmetric, that is, for any $2n$ hyperplanes $H_1,\dots,H_{2n}$ in $\CP^{n-1}$ and permutation $\sigma\in S_{2n}$ one has 
    \[
    \GonL_n(H_{1},\dots ,H_{2n})
    = (-1)^\sigma
    \GonL_n(H_{\sigma(1)},\dots ,H_{\sigma(2n)}).
    \]
    \item For any $2n+1$ hyperplanes $H_1,\dots,H_{2n+1}$ in $\CP^{n-1}$ one has 
    \[
    \sum_{j=1}^{2n+1} (-1)^j \GonL_n(H_1,\dots, \hat H_j ,\dots , H_{2n+1})=0.
    \]
    \item For any $2n+1$ hyperplanes $H_1,\dots,H_{2n+1}$ in $\CP^{n}$ one has 
    \begin{equation}\label{equ:gon2}
    \sum_{\substack{ j=1 \\ j \text{ admissible} } }^{2n+1} (-1)^j \GonL_n(H_j\cap H_1,\dots, \widehat{ H_j\cap H_j} ,\dots , H_j \cap H_{2n+1})=0,
    \end{equation}
    where the sum runs over the \emph{admissible} indices, and $j$ is admissible if $\cap_{i\neq j} H_i=\emptyset$ and there is no $i\neq j$ such that $H_i=H_j$.
\end{enumerate}
\end{thm}
In \eqref{equ:gon2} we tacitly identify $H_j$ with $\CP^{n-1}$. Since \eqref{equ:psi G def} is evidently invariant under the $\GL(n,\C)$-action on hyperplane configurations, the choice of this identification does not affect the value of $\GonL_n$.

\begin{rem}
    In the published version of \cite[Theorem 4.3]{Goncharov}, the admissibility condition appearing in \eqref{equ:gon2} was not present. This is an oversight confirmed by Goncharov \cite{GoncharovPersonal}, due to an overlooked boundary stratum that is relevant (only) in degenerate configurations for which $\cap_{i\neq j} H_i\neq \emptyset$. %The TODO \cite{GoncharovRudenko}...
\end{rem}

Goncharov has computed the value of $\GonL_n$ on some families of hyperplane configurations. Below we only need the following special case.

\begin{prop}[{Goncharov \cite[special case of Theorem 4.5]{Goncharov}}]
\label{prop:Goncharov nonvanishing}
Let $n\geq 3$ be odd and let $H_1,\dots, H_{2n}$ be the configuration of hyperplanes in $\CP^{n-1}$ defined in homogeneous coordinates $(z_1,\dots,z_n)$ as 
\begin{align*}
    z_1=0, z_2=0, \dots, z_n=0, z_1=z_2, z_2=z_3, \dots, z_{n-1}=z_n, z_n=z_1.
\end{align*}
Then 
\[
\GonL_n(H_1,\dots,H_{2n}) \neq 0.
\]
\end{prop}
\begin{proof}
    Goncharov considers a more general configuration where the points $X_1,\dots, X_{2n}$ in $\CP^{n-1}$ dual to our hyperplanes are such that $X_{n+k}$ is on the line through $X_k$ and $X_{k+1}$ ($k=1,\dots,n-1$) and $X_{2n}$ is on the line through $X_n$ and $X_1$.
    Then in \cite[Proposition 4.4]{Goncharov}, he assigns to this configuration a generalized cross-ratio, which is easily seen to be $a=1$ for our configuration.
    Finally, \cite[Theorem 4.5]{Goncharov} evaluates $\GonL_n(H_1,\dots,H_{2n})$ to a non-zero multiple of the polylogarithm $\mathrm {Li}_n(1)=\zeta(n)$.
\end{proof}

\subsection{The Goncharov cocycle $\alpha_{\mathrm G}$}
The first cocycle we consider is the \emph{Goncharov cocycle}.
It in fact factors through the complex sharbly Hopf algebra $\shB_\C$. That is, it is a composition 
\[
(\shB,d+D) \longrightarrow 
(\shB_\C,d+D) \longrightarrow \C.
\]
We will use the same notation $\alpha_{\mathrm G}\colon\shB_\C\to \C$ for the second map or for any restriction to $\shB_{\RR}$, $\RR\subset \C$.
The morphism $\alpha_{\mathrm G}$ acts on a rank-$n$ sharbly $[v_1,\dots,v_k]$ as

\[
\alpha_{\mathrm G}([v_1,\dots,v_k]) =
\begin{cases}
\GonL_{n}(v_1^+,\dots, v_k^+)
&\text{if $k=2n\geq 4$ and all $v_j\neq 0$} 
\\
0
&\text{otherwise}
\end{cases}.
\]

The following can be seen as merely a restatement of Goncharov's Theorem \ref{thm:Goncharov} and Proposition \ref{prop:Goncharov nonvanishing} above.
\begin{thm}
    \label{thm:alpha G}
The object $\alpha_{\mathrm G}$ is a cocycle, i.e., $\alpha_{\mathrm G}\circ (d+D) = 0$, and in fact $\alpha_{\mathrm G}\circ d=\alpha_{\mathrm G}\circ D=0$.
Furthermore, the image of the wheel graph $\alpha_{\mathrm G}(\Phi(W_n))$ is non-zero for any odd $n\geq 3$.
\end{thm}
\begin{proof}
We first check that $\alpha_{\mathrm G}$ is well defined.
The relation \eqref{equ:sh null rel} is satisfied by definition, and $\eqref{equ:sh perm rel}$ by antisymmetry of the Goncharov polylogarithm. The relation \eqref{equ:sh sign rel} is satisfied since the hyperplane $v_j^+$ does not depend on a rescaling of $v_j$.
Finally, consider the relation \eqref{equ:sh span rel}, i.e., assume that the $v_j$ do not span $\C^n$.
Then $\GonL_n=0$, because the integrand is a top form pulled back from $\CP^{n-2}$, and hence zero.

Next, we check that $\alpha_{\mathrm G}$ is a cocycle.
Precisely, we have 
\begin{align*}
\alpha_{\mathrm G}(d[v_1,\dots,v_k])
&=
\sum_{j=1}^k
(-1)^{j-1}
\alpha_{\mathrm G}([v_1,\dots,\hat v_j,\dots,v_k])
\\&=
\sum_{j=1}^k
(-1)^{j-1}
\GonL_{n}(v_1^+,\dots,\hat v_j^+,\dots,v_k^+)
=0
\end{align*}
for $k=2n+1$ by the second part of Theorem \ref{thm:Goncharov}.
Finally, we consider the deformed differential $D$. Recall that we have 
\[
D([v_1,\dots,v_k])
=
\sum_{j=1\atop\text{$v_j$ no coloop}}^k (-1)^{j}
[\bar v_1,\dots,\widehat{ \bar v}_j, \dots,\bar v_k]_{V/v_j}.
\]
Note that $v_j$ is a coloop iff $\cap_{i\neq j} v_i^+\neq \emptyset$.
Hence by the third part of Theorem \ref{thm:Goncharov} we have
\begin{align*}
\alpha_{\mathrm G}(D[v_1,\dots,v_k])
&=
\sum_{\substack{ j=1 \\ \text{$v_j$ no coloop}}}^k(-1)^{j}
\alpha_{\mathrm G}([\bar v_1,\dots,\widehat { \bar v}_j,\dots \bar v_k])
\\&=
\sum_{\substack{ j=1 \\ \text{$v_j$ no coloop} \\ \text{$j$ good} } }^k(-1)^{j}
\GonL_{n-1}(\bar v_1^+,\dots,\widehat { \bar v}_j^+,\dots \bar v_k^+)
\\&=
\sum_{\substack{ j=1 \\ \text{{$v_j$ no coloop}} \\ \text{$j$ good}}}^k(-1)^{j}
\GonL_{n-1}(v_1^+\cap v_j^+,\dots,\widehat {v_j^+\cap v_j^+},\dots v_k^+\cap v_j^+)
=0
\end{align*}
for $k=2n-1$. Here $j$ is good if there is no $i\neq j$ such that $v_i$ and $v_j$ are collinear, because in this case the quotient sharbly would have a zero vector. Note that the $j$ that are good and such that $v_j$ is no coloop are exactly the admissible $j$ in the sense of Theorem \ref{thm:Goncharov}.

Finally, for odd $n\geq 3$ consider the sharbly $\Phi(W_n)$ as in \eqref{equ:Phi Wn}. If we interpret the columns as the coefficient vectors of hyperplane equations, then the resulting hyperplane configuration is precisely that of Proposition \ref{prop:Goncharov nonvanishing}.
Hence by that proposition $\alpha_{\mathrm G}(\Phi(W_n))\neq 0$ as desired.
\end{proof}

For later use we also note:
\begin{lemma}\label{lem:alphaG Lie like}
    The cocycle $\alpha_{\mathrm G}$ is Lie-like, in the sense that
    \[
    \alpha_{\mathrm G}(xy) =0
    \]
    for all $x,y\in \shB_{\C}$ of positive degree.
\end{lemma}
\begin{proof}
    Let $x = [v_1,\dots, v_k] \in \Sh(\C^m)$ and $y=[w_1,\dots, w_\ell]\in \Sh(\C^n)$, with all $v_i,w_j\neq 0$ and $k,\ell\geq 1$. We may assume that $k+\ell=2m+2n$, for otherwise $\alpha_{\mathrm G}(xy)=0$ by degree reasons.
    In particular, $k+\ell\geq 4$ and we may assume by symmetry that $k\geq 2$. 
    Then we have to show that
    \begin{equation}\label{equ:agxy}
    0 = 
    \alpha_{\mathrm G}(xy) = 
    C
    \int_{\CP^{m+n-1}} 
\log\left|\frac{v_2\cdot z}{v_1\cdot z}\right|
\bigwedge_{\alpha=3}^{k} d\log\left|\frac{v_\alpha\cdot z}{v_1\cdot z}\right|
\bigwedge_{\beta=1}^{\ell} d\log\left|\frac{w_\beta\cdot z}{v_1\cdot z}\right|
,
    \end{equation}
    where $C$ is an unimportant nonzero constant and we have understood the $v_i$ and $w_j$ as elements of $\C^{m+n}\cong \C^m\oplus \C^n$.
    Consider the circle action on $\CP^{m+n-1}$ defined in homogeneous coordinates as
    \[
    (z_1,\dots,z_{m+n}) \longmapsto (z_1,\dots,z_m, e^{i\phi } z_{m+1},\dots,e^{i\phi} z_{m+n}),
    \] 
    and denote by $\xi$ its generating vector field.
    Denote the integrand of \eqref{equ:agxy} temporarily by $\omega$. It is a differential form of the top degree $2m+2n-2$, and satisfies $\iota_\xi\omega=0$, with $\iota_\xi$ the contraction with $\xi$. Since $\xi$ is non-vanishing on a dense subset this implies that $\omega=0$ on this subset, and hence the integral \eqref{equ:agxy} indeed vanishes.
    %The integrand above is a top form and is basic under this action and the generating vector field of the action is nonzero on a dense subset. Hence the integrand vanishes on a dense subset and the integral is zero. 
\end{proof}

\subsection{Consequences}
\begin{cor}\label{cor:Wn def injective}
    The morphism $\Phi$ of Theorem \ref{thm:main 2} sends the tadpole $t$ and the wheel classes $W_n$ for odd $n\geq 3$ to nonzero classes in $H(\shB, d+D)$. 
\end{cor}
\begin{proof}
    First consider the wheel classes.
    Since $W_n$ is a cycle in $(U^c\G_0^+,d+\nabla)$, we have that $\Phi(W_n)$ is a cycle in $(\shB, d+D)$.
    If $\Phi(W_n)$ were exact, i.e., $\Phi(W_n)=(d+D)x$ for some $x$, then $\alpha_{\mathrm G}(\Phi(W_n))=\alpha_{\mathrm G}((d+D)x)=0$, in contradiction to the last assertion of Theorem \ref{thm:alpha G}.

    We may proceed similarly for the tadpole class $t$. Here we consider the cocycle $\eta$ of Lemma \ref{lem:eta}. By the proof of that lemma, $\eta$ satisfies $\eta d=0$ and $(\eta\otimes \eta )\Delta=0$. By definition of $D$ as the action of $\eta$ the latter equation is equivalent to $\eta D =0$. Hence $\eta(d+D)=0$.
    Furthermore, $\eta(\Phi(t))=1$, so that we may conclude as before that $\Phi(t)$ is nonzero in $H(\shB, d+D)$.
\end{proof}

\begin{rem}\label{rem:wheel indep}
    Note that the wheel classes $W_n$ and also the tadpole $t$ are cycles also in the undeformed complex  $(U^c\G_0^+,d)$. Furthermore, by Theorem \ref{thm:alpha G} the Goncharov cocycle $\alpha_{\mathrm G}$ is also a cocycle with respect to the undeformed differential, i.e., $\alpha_{\mathrm G}\circ d=0$.
    Likewise, $\eta \circ d=0$ by the proof of Lemma \ref{lem:eta}.
    Hence the same proof as above also shows that $\Phi$ sends the $W_n$ ($n\geq 3$ odd) and $t$ to nonzero classes in $H(\shB, d)$, thus in particular giving an independent proof of Lemma \ref{lem:wheelnonzero} above.
\end{rem}

%\subsection{The Brown--Borel cocycle} 
%The Brown--Borel cocycle is defined on a sharbly $x=[v_1,\dots,v_k]$ of rank $n$ as 
%\begin{equation}\label{equ:alpha B def}
%\alpha_{B}(x)
%=
%\begin{cases}
%\int_{\Delta_{k-1}}\psi(x)^* \omega^{k-1}
%&\text{if $k=6,10,\dots$}
%\\
%0 &\text{otherwise}
%\end{cases},
%\end{equation}
%where $\psi(x)\colon \Delta_{k-1}\to X_n/\GL(n,\Z)$ is the relative cycle as in \eqref{equ:psi def}, and $\omega^{k-1}\in \Omega({P_n}/\GL(n,\Z))$ is the Borel form of degree $k-1$ as in \cite{BrownSchnetz}.
%Note that $\omega^{k-1}$ is defined only in the interior of the simplex and a priori has singularities on the boundary.
%The following result was essentially shown by Brown.
%\begin{thm}[{Adaptation of \cite[Theorem 14.13]{BrownBordification}}] \label{thm:alpha B}
%The integrals defining $\alpha_B$ converge and the resulting morphism is a cocycle, i.e., $\alpha_B\circ (d+D) = 0$.
%Furthermore, the image of the wheel graph $\alpha_{B}(\Phi(W_n))$ is non-zero for any odd $n\geq 3$.
%\end{thm}
%\begin{proof}[Proof sketch]
%The statement is essentially a special case of \cite[Theorem 14.13]{BrownBordification}. However, in loc. cit. it is formulated for a slightly different version of chains. Nevertheless, the proof goes through also for our chains that are in the image of $\psi$ in the same manner.
%\end{proof}

%Since we only we sketched the proof here, none of the other results of this paper will depend on the cocycle $\alpha_B$.

\subsection{Relation to the Rossi--Willwacher integrals}
\label{sec:RWvsG}
Another cocycle on $U^c\G_0^+$ has been described by Rossi--Willwacher \cite{RossiWillwacher}, as a side product of their computation of a family of Drinfeld associators.
%This cycle $\alpha_{\mathrm RW}$ is zero unless the number of edges is equal to $2v-2$, where $v$ is the number of vertices, and itassociates to a graph $\Gamma$ with $v$ vertices, $k$ edges $e_1=(a_1,b_1),\dots,e_k=(a_k,b_k)$ and no tadpoles the number 
This cocycle $\alpha_{\mathrm RW}$ associates to a graph $\Gamma$ with $v$ vertices and $k$ edges the number 
\[
\alpha_{\mathrm RW}(\Gamma)=
\int_{\mathrm{Conf}_v(\C)/\C^\times \ltimes \C}
\sum_{i=1}^k
\sum_{j\neq i}
(-1)^{\epsilon(i,j)}
\log|z_{a_j}-z_{b_j}|
\bigwedge_{\alpha=1\atop \alpha \neq i,j}^k
d\arg(z_{a_\alpha}-z_{b_\alpha}).
\]
If $k\neq 2v-2$ the integrand is not a top degree form and $\alpha_{\mathrm RW}(\Gamma)=0$.
Here the integral is over the quotient of the configuration space of $v$ points in $\C$ by translations, rotation and scaling, with $(z_1,\dots,z_v)$ being the natural coordinates of a configuration of points.
The sign is 
\begin{equation}\label{equ:eps ij}
(-1)^{\epsilon(i,j)} 
=
\begin{cases}
    (-1)^{i+j-1} & \text{if $i<j$} \\
    (-1)^{i+j} & \text{if $i>j$}
\end{cases}
\end{equation}
and the argument form is
\[
d\arg(z) = \frac{1}{2i}\left( \frac{dz}{z} - \frac{d\bar z}{\bar z}\right).
\]
In case $\Gamma$ has a tadpole or is the empty graph, we set $\alpha_{\mathrm RW}(\Gamma)=0$. 
Note that in order to have a nonzero integral there must be as many holomorphic as antiholomorphic form components, since we are integrating over a complex manifold. It follows that we may equivalently replace $d\arg(z)$ by $d\log|z|$ in the integrand, so that
\[
\alpha_{\mathrm RW}(\Gamma)=
C
\int_{\mathrm{Conf}_v(\C)/\C^\times \ltimes \C}
\sum_{i=1}^k
\sum_{j\neq i}
(-1)^{\epsilon(i,j)}
\log|z_{a_j}-z_{b_j}|
\bigwedge_{\alpha=1\atop \alpha \neq i,j}^k
d\log|z_{a_\alpha}-z_{b_\alpha}|,
\]
for $C$ some unimportant constant that depends on $k$, but not on $\Gamma$.
Furthermore, it is convenient to eliminate the quotient from the integration domain by fixing the positions of the points $z_1=0$ and $z_2=1$. We may furthermore assume that vertices 1 and 2 in our graph $\Gamma$ are connected by edge $e_1=(1,2)$, so that the only nonzero remaining terms in the integral are
\begin{equation}\label{equ:RW simplified}
\alpha_{\mathrm RW}(\Gamma)=C
\int_{(z_3,\dots,z_v)\in \C^{v-2}}
\sum_{j=2}^k
(-1)^{j}
\log|z_{a_j}-z_{b_j}|
\bigwedge_{\alpha=2\atop \alpha \neq j}^k
d\log|z_{a_\alpha}-z_{b_\alpha}|.
\end{equation}
We will use below:
\begin{lemma}\label{lem:alphaRW bridge}
    If the graph $\Gamma\in U^c\G_0^+$ is not connected, or if it has a bridge then $\alpha_{\mathrm RW}(\Gamma)=0$.
\end{lemma}
\begin{proof}
    If $\Gamma$ has a tadpole, then $\alpha_{\mathrm RW}(\Gamma)=0$, so we may assume that $\Gamma$ has no tadpole.
    Consider next the case that $\Gamma$ has a bridge.
    We may assume by renumbering that our edge $e_1=(1,2)$ is a bridge. Removing $e_1$ hence disconnects $\Gamma$ into two disjoint subgraphs $\Gamma_1$ and $\Gamma_2$, containing vertex 1 (resp. 2), and each containing more than one vertex. Let us assume $\Gamma_1$ contains the vertices $1,3,\dots,r$. Consider the circle action 
    \[
    (z_3,\dots,z_v) \longmapsto (e^{i\phi}z_3,\dots,e^{i\phi}z_r, z_{r+1},\dots, z_v), 
    \]
    and denote by $\xi$ the corresponding generating vector field on $\C^{v-2}$. The vector field $\xi$ is non-vanishing on a dense subset and its contraction with the integrand of \eqref{equ:RW simplified} vanishes. Hence $\alpha_{\mathrm RW}(\Gamma)=0$, cf. the argument used in proof of Lemma \ref{lem:alphaG Lie like}.
    %Denote the integrand of \eqref{equ:RW simplified} $\alpha_{\mathrm RW}(\Gamma)=0$.
    %temporarily by $\omega$. It is a differential form of the top degree $2v-4$, and satisfies $\iota_\xi\omega=0$, with $\iota_\xi$ the contraction with $\xi$. Since $\xi$ is non-vanishing on a dense subset this implies that $\omega=0$ on this subset, and hence $\alpha_{\mathrm RW}(\Gamma)=0$.$\alpha_{\mathrm RW}(\Gamma)=0$.
    
    Finally assume $\Gamma$ is not connected, say $\Gamma=\Gamma_1\sqcup \Gamma_2$ for non-empty subgraphs $\Gamma_1$ and $\Gamma_2$, necessarily both containing more than one vertex. Pick edge $e_1=(1,2)$ in $\Gamma_1$. Then the integrand is basic under the circle action on the coordinates corresponding to the vertices of $\Gamma_2$, and we conclude as before that $\alpha_{\mathrm RW}(\Gamma)=0$. We also note that (essentially) the same argument shows more generally that $\alpha_{\mathrm RW}(\Gamma)=0$ if $\Gamma$ is not biconnected.
\end{proof}

On the other hand, recall from \eqref{equ:psi G def} the function $\psi_n$ that defines, up to a nonzero prefactor $\alpha_{\mathrm G}([v_1,\dots,v_{k}])$.
Let us rewrite the expression \eqref{equ:psi G def} in a form closer to \eqref{equ:RW simplified}.
By the first part of Theorem \ref{thm:Goncharov} the function $\psi_n(v_1,\dots,v_{2n})$ is in fact antisymmetric in the arguments, so we may antisymmetrize without changing the function. We may furthermore pick explicit complex coordinates for the integration by setting $z_1=1$. Finally, we may assume we have chosen our coordinates so that $v_1\cdot z = z_1$. With these simplifications \eqref{equ:psi G def} becomes
\begin{equation}\label{equ:psi G 2}
\psi_n(v_1,\dots,v_{2n})
=
\frac 1{2n-1}
\int_{(z_2,\dots,z_n)\in \C^{n-1}} 
\sum_{j=2}^{2n}(-1)^{j}
\log|v_j\cdot z|
\bigwedge_{\alpha=2\atop \alpha\neq j}^{2n} d\log|v_\alpha\cdot z|.
\end{equation}

Comparing \eqref{equ:RW simplified} and \eqref{equ:psi G 2} (see also \eqref{equ:psi is L}), we arrive at the following result:
\begin{prop}\label{prop:alpha RW G}
There are non-zero constants $C_n$ with the following property.
Let $\Gamma\in U^c\G_0$ be a graph of loop order $\beta_1(\Gamma)=n$.
Then we have 
\begin{equation}\label{equ:alpha RW G}
\alpha_{\mathrm RW}(\Gamma) = C_n \alpha_{\mathrm G}(G(\Phi(\Gamma))),
\end{equation}
where $G\colon\shB_\Q\to \shB_\Q$ is the Gale duality morphism of Section \ref{sec:gale}.
\end{prop}
\begin{proof}
Let $\Gamma$ be a graph with $k$ edges and $p$ vertices.
If $p=0$ both sides of \eqref{equ:alpha RW G} are zero.
If $\Gamma$ has a tadpole, then $\alpha_{\mathrm RW}(\Gamma)=0$ by definition and $\alpha_{\mathrm G}(G(\Phi(\Gamma)))=0$ because the tadpole contributes a coloop to $\Phi(\Gamma)$ which becomes a null vector in $G(\Phi(\Gamma))$.
If $\Gamma$ has a bridge, then $G(\Phi(\Gamma))=0$, and $\alpha_{\mathrm RW}(\Gamma)=0$ by Lemma \ref{lem:alphaRW bridge}. 
So assume that $\Gamma$ is non-empty, bridgeless and tadpole-free.
We may also assume throughout that the numbering of vertices and edges has been chosen so that vertices 1 and 2 are connected by the first edge.
Note that by Example \ref{ex:Gale dual graph} the sharbly 
\[
x:=[v_1,\dots,v_k] := G(\Phi(\Gamma))
\]
is (essentially) the edge-vertex incidence matrix of the graph $\Gamma$, or more precisely the matrix of the boundary operator $\Q^k\cong C_1(\Gamma)\to \ker \iota^* \subset C_0(\Gamma)$.
So we have 
\[
\alpha_{\mathrm G}(G(\Phi(\Gamma)))= \begin{cases}
    C_{n'} \psi_{n'}(v_1,\dots,v_k) & \text{if $k=2n'$} \\
    0 & \text{otherwise}
\end{cases},
\]
with $C_{n'}$ a non-zero constant and
\[n'=\vdim \ker \iota^* = \underbrace{\vdim C_0(\Gamma)}_{=p}-\vdim H_0(\Gamma).\]

Hence $\alpha_{\mathrm G}(G(\Phi(\Gamma)))$ is only non-zero if $k=2n'$. Also, if $\Gamma$ was not connected, then $x$ would have a block form in the appropriate basis
\[
x= \begin{pmatrix}
    * & 0 \\
    0 & *
\end{pmatrix},
\]
and hence be a nontrivial product, so that $\alpha_{\mathrm G}(G(\Phi(\Gamma)))=0$ by Lemma \ref{lem:alphaG Lie like}.
%and the integrand of \eqref{equ:psi G 2} is basic under rotation ($z\mapsto e^{i\phi}z$) of the coordinates corresponding to the second block. Hence in this case $\psi_{n'}(v_1,\dots,v_k)=0$. 
%(Alternatively, we could have 
So $\alpha_{\mathrm G}(G(\Phi(\Gamma)))$ is nonzero only for $\Gamma$ connected and in that case $n=n'$, $k=2n$ and $p=n+1$.

On the other hand, observe that by degree counting the expression $\alpha_{\mathrm RW}(\Gamma)$ is zero unless $k=2p-2$. Furthermore, if $\Gamma$ is not connected, then $\alpha_{\mathrm RW}(\Gamma)=0$ by Lemma \ref{lem:alphaRW bridge}. But in the connected case we again have $n=k-p+1=p-1$ and hence $k=2n$.

Thus we may assume throughout that $\Gamma$ is a connected graph with $k=2n=2n'$ edges and $p=n+1$ vertices, for otherwise both sides of \eqref{equ:alpha RW G} are zero. 
But then, using our assumption that vertices 1 and 2 are connected by the first edge, we see that \eqref{equ:RW simplified} and \eqref{equ:psi G 2} are identical, up to an unimportant prefactor. From \eqref{equ:psi is L} the proposition hence follows for $\Gamma$.
\end{proof}

\newcommand{\BC}{{\mathbf{B}}}
\newcommand{\ft}{{\mathfrak{t}}}
\section{Proof of Theorem \ref{thm:PID lower}}

\subsection{A morphism of dg Lie algebras}
We will use the complex dg Hopf algebra $\BC:= \shB_\C \otimes \C$, with the undeformed differential $d$. Following \eqref{equ:fgA def} define the associated dg Lie algebra 
\[
\fg_\BC := Q(\BC)^*.
\]
We already know the following closed elements in this dg Lie algebra:
\begin{itemize}
    \item The element $\eta$ of cohomological degree $1$ of Section \ref{sec:deformed D}, which by Lemma \ref{lem:eta} is a cocycle in $\fg_\BC$. 
    \item The parts $\alpha^{m}_G$ of the Goncharov cocycle of homogeneous degree $m$. They are in $\fg_\BC$ by Lemma \ref{lem:alphaG Lie like} and closed by Theorem \ref{thm:alpha G}.
\end{itemize}

Also consider the graded Lie algebra 
\[
\ft  := \C T \oplus \FreeLie(\sigma_3,\sigma_5,\dots)
\]
where $T$ is a generator of degree $1$, and $\sigma_m$ for $m\geq 3$ odd is a generator of degree $2m$.

\begin{lemma}
    The morphism 
    \[
    f \colon \ft \longrightarrow \fg_\BC
    \]
    defined on generators as $f(T)=\eta$, $f(\sigma_m)=\alpha^{2m}_G$
    is a morphism of dg Lie algebras.
\end{lemma}
\begin{proof}
    Since $\eta$ and the $\alpha^{2m}_G$ are cocycles, it is clear that $f$ respects the differentials.
    It then suffices to check the commutation relations 
    \begin{align*}
    [\eta,\eta] &=0
    & 
    [\eta, \alpha^{2m}_G] &=0.
    \end{align*}
    The former relation has been verified in Lemma \ref{lem:eta}. The latter is equivalent to $\alpha^{2m}_G\circ D=0$ and has been verified in Theorem \ref{thm:alpha G}.
\end{proof}

\subsection{Morphism to the graph complex}
Our morphism $\Phi\colon U^c \G_0^+ \to \shB$ may be prolonged by the homomorphism of dg Hopf algebras $\shB:=\shB_{\Z} \to \shB_{\C}$, coming from naturality under ring inclusions of the sharbly construction.
Complexifying and passing to the dual indecomposables, we obtain a morphism of complex dg Lie algebras
\[
\Phi^* \colon \fg_{\BC} \longrightarrow \GC_0^+ \otimes \C.
\]
This morphism sends $\eta$ to $T/2\in \GC_0^+\otimes \C$ (the tadpole graph).
We already know by Proposition \ref{prop:Goncharov nonvanishing} that $\Phi^*(\alpha_{\mathrm G}^{2m})\in \GC_0^+\otimes \C$ is a linear combination of graphs that contains the dual wheel graph $W_m^*\in \GC_0^+\otimes \C$ with non-zero coefficient, for each odd $m\geq 3$. The following observation is then immediate:
\begin{cor}\label{cor:Phiprop}
    The morphism of dg Lie algebras $\Phi^*\circ f\colon \ft \to \GC_0^+\otimes \C$ sends the generator $T\in \ft$ to $T/2\in \GC_0^+\otimes \C$, and the generator $\sigma_m$ to a linear combination of graphs that contains the dual wheel graph $W_m^*$ with non-zero coefficient.
\end{cor}

\subsection{The Grothendieck--Teichmüller group}
The Grothendieck--Teichmüller Lie algebra $\grt_1$ is a well-studied infinite dimensional Lie algebra introduced by Drinfeld \cite{Drinfeld} in his study of the absolute Galois group of $\Q$. We shall only need the complex version.
It is the subspace 
\[
\grt_1 \subset \FreeLie( X,Y)
\]
of the free Lie algebra in generators $X$ and $Y$ satisfying a set of linear equations. The precise form of these equations will be irrelevant for us, see \cite{WillwacherFreeness} for a recollection. The Lie bracket is the Ihara bracket, which differs from the Lie bracket on the underlying free Lie algebra.
We shall need only the following properties of $\grt_1$.
First, $\grt_1$ is graded by weight, with the weight of a homogeneous element being equal to the number of generators $X$ and $Y$ appearing. We denote the weight $W$ part by $\grt_1^W\subset \grt_1$.
Second, $\grt_1$ computes part of the homology of the graph complex. The following theorem relates $\grt_1$ to the cohomology of the graph complex $\GC_0^{k-\mathrm{loop}}$ with fixed loop order $k$.

\begin{thm}[\cite{WillwacherGrt}]
\label{thm:Wgrt}
For each $k\geq 0$ we have that 
\[
H^{2k}(\GC_0^{k\mathrm{-loop}})\otimes \C \cong \grt_1^{k}.
\]
The isomorphism sends a graph cocycle with non-zero coefficient of the dual wheel graph $W_{2j+1}^*$ to a Lie word with non-zero coefficient in front of the Lie monomial 
\[
\ad_X^{2j} Y := [X,[X,[\cdots,[X,Y]\cdots]]] \quad \text{($2j$ brackets).}
\]
Furthermore, $H^{<2k}(\GC_0^{k\mathrm{-loop}})=0$.
\end{thm}

Furthermore, the following theorem has been shown by Brown, with a minor addition of the second author.

\begin{thm}[Brown \cite{Brown}, Willwacher \cite{WillwacherFreeness}]
\label{thm:W freeness}
Let $\sigma_{2j+1}\in \grt_1$, $j=1,2,\dots$, be any collection of homogeneous elements of weight $2j+1$ such that the coefficient of $\ad_X^{2j} Y$ in $\sigma_{2j+1}$ is non-zero.
The Lie subalgebra of $\grt_1$ generated by this family is free.
\end{thm}

\subsection{Combined map and injectivity}

From Corollary \ref{cor:Phiprop} and Theorems \ref{thm:Wgrt} and \ref{thm:W freeness} we immediately conclude:

\begin{cor}\label{cor:inj}
    The composition of morphisms of graded Lie algebras 
    \[
    \ft \xrightarrow{f} H(\fg_{\BC}) \xrightarrow{\Phi^*} H(\GC_0^+) \otimes \C
    \]
    is injective.
\end{cor}

\subsection{Dual morphisms of Hopf algebras}
Consider again the morphisms of dg Lie algebras 
\[
 \ft \xrightarrow{f} \fg_\BC := Q(\BC)^*.
\]
Since $\ft$ is degree-wise finite dimensional, this is the same data as a morphism of dg Lie coalgebras 
\[
 Q(\BC) \longrightarrow \ft^*,
\]
and this in turn is the same data as a morphism of dg Hopf algebras
$\BC \to U^c\ft^*$. Prolonging this latter morphism by the map $\Phi$ and taking homology we obtain from Corollary \ref{cor:inj}:

\begin{cor}\label{cor:surj}
The composition of morphisms of graded Hopf algebras 
\[
H(U^c\G_0^+)\otimes \C \longrightarrow 
H(\BC) \longrightarrow U^c\ft^*
\cong \Lambda_\C ( t) \otimes T_\C(s_3,s_5,\dots)
\]
is surjective.
\end{cor}
\begin{proof}
    We yet have to check that the composition $H(U^c\G_0^+)\otimes \C \cong U^cH(\G_0^+)\otimes \C \to U^c\ft^*$ is surjective. By Poincar\'e--Birkhoff--Witt this is equivalent to surjectivity of the morphism of dg Lie coalgebras $H(\G_0^+)\otimes \C \to \ft^*$. This in turn is equivalent to injectivity of the dual map, which holds by Corollary \ref{cor:inj}.
\end{proof}

\subsection{Proof of Theorem \ref{thm:PID lower}}
Let $\RR$ be a PID with an injective morphism of commutative rings $\RR\to \C$. Automatically, the morphism $\Z\to \C$ then factors through $\RR$.
By the naturality of the sharbly construction under injective ring morphisms this induces morphisms of dg Hopf algebras 
\[
\shB = \shB_\Z \longrightarrow \shB_\RR \longrightarrow \shB_\C.
\]
Taking homology, complexifying and prolonging this chain we obtain the morphisms of complex graded Hopf algebras 
\[
H(U^c\G_0^+)\otimes \C \longrightarrow 
\shA_\Z \otimes \C \longrightarrow
\shA_\RR \otimes \C \longrightarrow
\shA_\C \otimes \C \longrightarrow U^c(\ft^*)
\cong \Lambda_\C ( t) \otimes T_\C(s_3,s_5,\dots).
\]
By Corollary \ref{cor:surj} the composition of these morphisms is surjective. Hence so is in particular the composition of the last two arrows, which yields Theorem \ref{thm:PID lower}.
\hfill\qed

\section{Proof of Theorem \ref{thm:main 2}}

We have so far constructed the deformed differential and shown the first part of Theorem \ref{thm:main 2}.
It remains to check the second part, i.e., that the morphism of dg Hopf algebras \eqref{equ:Phi deformed} is a quasi-isomorphism.
We also know from Corollary \ref{cor:KWZ} the homology of the left-hand side of \eqref{equ:Phi deformed}.
In the following subsection we will compute the homology of the right-hand side and see that it is isomorphic. Afterwards, we verify that the morphism $\Phi$ induces a homology isomorphism.

%We already know from \cite{KW} that 
%\[
%H(\G_0,d+\nabla) \cong \vspan_{\Q}( t, W_{3}, W_{5},\dots ),
%\]
%with $t$ the generator dual to $T$ in degree $+1$ and $W_n$ the wheel graph with $n$ spokes and hence $2n$ edges in degree $2n$.
%It follows that for the universal enveloping coalgebra we have
%\[
%H(U^c\G_0,d+\nabla) \cong \Q[ t, W_{3}, W_{5},\dots ]
%\cong H(BK(\Z), \Q).
%\]

%\subsection{The Waldhausen construction and K theory of $\Z$}

\subsection{The Waldhausen construction for \texorpdfstring{$\Proj(\Z)$}{Proj(Z)}}
\label{sec:waldhausen-recollection}

Let
\[
  \mathcal P:=\Proj(\Z)
\]
be the category of finitely generated projective $\Z$-modules.  We use its standard Waldhausen structure: the weak
equivalences are the isomorphisms, and the cofibrations are the injective
maps with projective cokernel. 
% Since $\Z$ is a PID, every object of
%$\mathcal P$ is free and every such cofibration is split, although no splitting
%is included in the data.

For $p\geq 0$, let $\operatorname{Ar}[p]$ be the arrow category of the ordered
set $[p]=\{0<\cdots<p\}$. Its objects are pairs $(a,b)$ with
$0\leq a\leq b\leq p$.  Waldhausen's category $S_p\mathcal P$ \cite[Section~1.3]{Waldhausen} is the full
subcategory of the functor category $\operatorname{Fun}(\operatorname{Ar}[p],\mathcal P)$ consisting
of diagrams $A=(A_{a,b})$ satisfying
\[
  A_{a,a}=0
\]
and such that, for every $a\leq b\leq c$, the structure maps form a
cofibration sequence
\begin{equation}\label{eq:waldhausen-cofibration-sequence}
  0\longrightarrow A_{a,b}\longrightarrow A_{a,c}
   \longrightarrow A_{b,c}\longrightarrow 0.
\end{equation}
Equivalently, an object is a flag of cofibrations
\begin{equation}\label{equ:Wal chain}
  0=A_{0,0}\lhook\joinrel\longrightarrow A_{0,1}
  \lhook\joinrel\longrightarrow\cdots
  \lhook\joinrel\longrightarrow A_{0,p},
\end{equation}
together with the specified quotients $A_{a,b}\cong A_{0,b}/A_{0,a}$.
Let $wS_p\mathcal P\subset S_p\mathcal P$ be the maximal subgroupoid -- i.e., we keep only the object-wise isomorphisms as morphisms.

An order-preserving map $\alpha\colon [m]\to[p]$ induces a pullback functor
$\alpha^*\colon wS_p\mathcal P\to wS_m\mathcal P$. In this way
$p\mapsto wS_p\mathcal P$ is a simplicial groupoid. Applying the nerve in
the groupoid direction gives the bisimplicial set
\begin{equation}\label{eq:waldhausen-bisimplicial-set}
  \mathcal W_{p,q}:=N_qwS_p\mathcal P.
\end{equation}
We define
\begin{equation}\label{eq:BK-waldhausen}
  BK(\Z):=\bigl|\operatorname{diag}\mathcal W\bigr|.
\end{equation}
We write
\[
  C^{\mathrm{Wal}}
  :=\operatorname{Tot}_{p,q}N_p^\Delta N_q^\Delta\mathcal W\otimes\Q
\]
for the normalized total chain complex, where $N^\Delta$ denotes
normalization in the indicated simplicial direction. (Note that we need to distinguish notationally the normalized chains and the nerve functors.) The differential has the form 
\[
\partial_{\rm Wal} = \partial_S + (-1)^p \partial_{\rm nerve},
\]
where $\partial_S=\sum_{j=0}^p(-1)^{j-1}\partial_j$ arises from the simplicial structure on $wS_\bullet \mathcal P$ and $\partial_{\rm nerve}$ is the nerve differential. Note that we shall make the Koszul signs explicit in this section, so that $\partial_{\rm nerve}$ acts on a chain of simplicial degree $p$ in the first direction with a sign $(-1)^p$ as indicated in the formula.

The chain complex $C^{\mathrm{Wal}}$ is filtered by the rank of the
upper-right object:
\begin{equation}\label{eq:waldhausen-rank-filtration}
  F_r wS_p\mathcal P
  :=\{A\in wS_p\mathcal P\mid\operatorname{rk}A_{0,p}\leq r\}.
\end{equation}
This is the rank filtration used in Quillen's spectral sequence, see \cite{QuillenSpSq} and e.g. \cite[Proposition 6.1]{DEKM}. 

\subsection{The cubical configuration construction}
\label{sec:cubical-waldhausen-construction}
\subsubsection{Enhanced ordered configurations}
We now define a cubical groupoid whose cubical boundary contains both the
ordinary sharbly differential $d$ and the rank-changing operator $D$.
Let $[k]=\{1,\ldots,k\}$ and for $V\in \mathcal P$ put
\[
  \overline V:=V/(v\sim-v).
\]
Note that we allow the distinguished element $0\in\overline V$ -- it will be used for cubical degeneracies later. 
For an ordered family
$\mathbf v=(\bar v_1,\ldots,\bar v_k)\in\overline V^k$, denote the saturated span of the vectors $v_i$, $i\in I\subset [k]$ by
\begin{equation}\label{eq:saturated-span-configuration}
  P_I=P_I(\mathbf v)
  :=V\cap\operatorname{span}_{\Q}\{v_i\mid i\in I\},
  \qquad I\subset [k].
\end{equation}

Each $P_I$ is saturated in $V$, and if $I\subseteq J$, then
$P_I$ is saturated in $P_J$; hence $P_J/P_I$ belongs to $\mathcal P$.
Let
\[
  \Lambda(\mathbf v):=\{P_I\mid I\subseteq [k]\}
\]
be the finite poset of saturated spans, and let
$\operatorname{Ar}(\Lambda(\mathbf v))$ be its arrow category. 

\newcommand{\Ar}{\operatorname{Ar}}
\begin{definition}\label{def:enhanced-ordered-configuration}
An \emph{enhanced ordered configuration of length $k$} is a triple
\[
  F=(V,\mathbf v,\mathcal V),
\]
where $V$ is a finitely generated projective $\Z$-module, $\mathbf v\in \bar V^{\times k}$ is a $k$-tuple of elements of $\bar V$ that span $V\otimes \Q$ and 
\[
  \mathcal V\colon\Ar(\Lambda(\mathbf v))\longrightarrow\mathcal P
\]
is a functor, satisfying the following properties:
\begin{enumerate}
    \item For each $A\subset B\subset C$ in $\Lambda(\mathbf v)$ the sequence of structure maps
    \[
    0\longrightarrow \mathcal V(A,B) \longrightarrow \mathcal V(A,C) \longrightarrow \mathcal V(B,C) \longrightarrow 0
    \]
    is exact.
    \item $\mathcal V(0,V)=V$ and $\mathcal V(A,A)=0$ for each $A\in \Lambda(\mathbf v)$.
    \item For each $A\in \Lambda(\mathbf v)$ the map 
    \[\mathcal V(0,A) \longrightarrow \mathcal V(0,V)=V\]
    is an isomorphism onto $A\subset V$.
\end{enumerate} 
\end{definition}

In particular, the axioms imply that $\mathcal V(A,B) \cong B/A$, with a canonical isomorphism encoded by the data.
For $I\subseteq J\subseteq [k]$, put
\[
  F_{I,J}:=\mathcal V(P_I,P_J).
\]
Thus $F_{I,J}=F_{I',J'}$ whenever $P_I=P_{I'}$ and $P_J=P_{J'}$.

A morphism of enhanced ordered configurations from $(V,\mathbf v,\mathcal V)$ to
$(W,\mathbf w,\mathcal W)$ consists of an isomorphism
$f\colon V\longrightarrow W$ carrying $\mathbf v$ to $\mathbf w$, together with a natural family of isomorphisms
\[
  \alpha_{A,B}\colon
  \mathcal V(A,B)\xrightarrow{\cong}\mathcal W(f(A),f(B))
\]
whose component $\alpha_{0,V}$ equals $f$.
We denote the
resulting groupoid of enhanced ordered configurations of length $k$ by $\mathcal E_k$.

\subsubsection{Discussion and motivation}
Before continuing, let us pause and discuss the meaning behind this definition, and compare it to the Waldhausen construction.
Objects in the Waldhausen category are flags of modules as in \eqref{equ:Wal chain} together with fixed choices of the quotients $A_{a,b}\cong A_{0,b}/A_{0,a}$. One could argue that this data is irrelevant because we could just set $A_{a,b}:=A_{0,b}/A_{0,a}$. However, the relations required by the simplicial structure would then not necessarily hold -- or rather they would hold up to isomorphism, but not with equality as is required. Hence, to fix the problem, an object in the Waldhausen category has to carry along a fixed repository of representatives for the isomorphism classes of all possible quotients, along with the explicit isomorphism to the quotient.

The idea behind the enhanced ordered configurations is to do the same for the cubical version. Namely, along with the collection of vectors $\mathbf v$ in $\bar V$ we store a repository of explicit representatives for all possible quotients, together with an isomorphism to the quotient.
We call this data $\mathcal V$ the enhancement.
For a fixed
$(V,\mathbf v)$, its groupoid of enhancements is contractible: any two
choices are joined by the unique isomorphism obtained from their
canonical identifications with $B/A$.
We also remark that we index the choices of quotients not by pairs of subsets of $[k]$, but by the pairs of the corresponding subspaces of $V$. This will be convenient shortly.

\subsubsection{Cubical structure}
For $j\in [k]$, define two face functors
\[
  C_j,R_j\colon\mathcal E_k\longrightarrow\mathcal E_{k-1}.
\]
Fix an object $F=(V,\mathbf v, \mathcal V)$ with $\mathbf v=(v_1,\dots,v_k)$ and let $S=[k]\setminus\{j\}$. We will use subsets of $S$ to index elements of the image of $F$ below, with the understanding that afterwards we re-index to $[k-1]\cong S$.
\begin{itemize}
    \item Let $F'=(V',\mathbf v', \mathcal V'):= C_j F$. Then the new ambient module is $V'=F_{\{j\},[k]}\cong V/P_{\{j\}}$. Let 
        \[
        q_j \colon V=F_{\emptyset, [k]} \longrightarrow F_{\{j\},[k]}=V'
        \]
        be the structure map. The new marked vectors are the $q_j(v_i)$ for $i\in S$, i.e., $v_j$ is omitted and the remaining vectors are transported along $q_j$.
        Their saturated spans are 
        \begin{align*}
        P_I' &= q_j(P_{I\cup \{j\}})
        =
        \im(F_{\{j\},I\cup \{j\}}\longrightarrow V')\quad (\text{for $I\subseteq S$}).
        \end{align*}
        Note that 
        \begin{equation}\label{equ:subs eq well1}
        P_{I}' \subseteq P_J' \iff 
        P_{I\cup \{j\}} \subseteq P_{J\cup \{j\}}
        \end{equation}
        and likewise for equality.
        This means that we obtain a well-defined functor 
        \begin{gather*}
            Q_j \colon \Lambda(\mathbf v') \longrightarrow \Lambda(\mathbf v) \\
            P_I' \longmapsto P_{I\sqcup \{j\}} \quad (\text{for $I\subseteq S$}).
        \end{gather*}
        We then define the functor $\mathcal V'$ as the composition  
        \[
        \Ar(\Lambda(\mathbf v')) \xrightarrow{\Ar Q_j} \Ar(\Lambda(\mathbf v)) \xrightarrow{\mathcal V} \mathcal P.
        \]
    \item Similarly, let $F'=(V',\mathbf v', \mathcal V'):= R_j F$. Then the new ambient module is $V'=F_{\emptyset,S}$ with $S=[k]\setminus\{j\}$. Let
        \[
        u_j \colon V' = F_{\emptyset,S} \longrightarrow F_{\emptyset,[k]}=V
        \]
        be the structure map. It is injective with image $P_S$. 
        The new marked vectors are the unique lifts $v_i'$ of the $v_i$ under $u_j$, for $i\in S$.
        Their saturated spans are 
        \[
        P_I' = u_j^{-1}(P_I)= \im(F_{\emptyset, I} \longrightarrow V')\quad \text{for } I\subseteq S.
        \]
        Hence 
        \[
        P_{I}' \subseteq P_J' \iff 
        P_{I} \subseteq P_{J}
        \]
        and likewise for equality. Therefore, we obtain again a functor 
        \begin{gather*}
            U_j \colon \Lambda(\mathbf v') \longrightarrow \Lambda(\mathbf v) \\
            P_I' \longmapsto P_{I} \quad (\text{for $I\subseteq S$}),
        \end{gather*}
        and define the functor $\mathcal V'$ as the composition  
        \[
        \Ar(\Lambda(\mathbf v')) \xrightarrow{\Ar U_j} \Ar(\Lambda(\mathbf v)) \xrightarrow{\mathcal V} \mathcal P.
        \]
\end{itemize}
The axioms of an enhanced ordered configuration for $C_jF, R_jF$ are immediately inherited from $F$.
The definition of $C_j, R_j$ on (iso)morphisms is also evident.

Inserting a zero vector defines a degeneracy
$s_j\colon \mathcal E_{k-1}\to\mathcal E_k$. 
Since all faces are restrictions of
one fixed diagram, the cubical identities hold strictly. 

\begin{rem}
It is here that we need our repository of fixed representatives of all quotients. For example, one of the cubical identities is, for $i<j$,
\[
C_i C_j = C_{j-1} C_i.
\]
A naive construction would assign on the left the ambient object $V$ quotiented by $v_j$ then $v_i$, and on the right $V$ quotiented by $v_i$ then $v_j$. Although both are evidently isomorphic, they are not equal, so the above equality of functors would not hold. Our construction solves this by assigning instead at both sides the same representative for the quotient from the repository.
Moreover, if the $j$th vector is zero, then $P_{I\cup\{j\}}=P_I$ and hence $C_jF=R_jF$ literally. 
Here we used that we index the repository of quotients by the submodules $P_I, P_J$ and not just the collection of indices $I,J$.
\end{rem}

\subsubsection{The cubical/simplicial chain complex}
Apply the nerve in the groupoid direction to obtain $N_q\mathcal E_k$.
This is a cubical object in simplicial sets.  Its normalized cubical--simplicial
total complex with rational coefficients is
\begin{equation}\label{eq:standard-cubical-total-complex}
  X^\square
  :=\operatorname{Tot}_{k,q}N_k^\square N_q^\Delta N_\bullet(\mathcal E_\bullet),%\otimes\Q,
\end{equation}
with cubical differential
\begin{equation}\label{eq:standard-cubical-boundary}
  \partial_\square
  =\sum_{j=1}^k(-1)^{j-1}(R_j-C_j).
\end{equation}
The total differential is the cubical differential plus the nerve differential
\begin{equation}\label{equ:partial tot}
\partial_{\rm tot} = \partial_\square + (-1)^k \partial_{\rm nerve}.
\end{equation}
The bigraded piece $X_{k,q}^\square$ carries a natural action of the symmetric group $S_k$. We define the complex of coinvariants 
\begin{align}\label{equ:X def}
X &:= \bigoplus_{k,q} X_{k,q} &
X_{k,q} &:=  (X_{k,q}^\square \otimes \mathrm{sgn}_k)_{S_k}.
\end{align}
The differential \eqref{equ:partial tot} evidently descends to the coinvariants, and we continue to use the same symbol for the differential on $X$.

\begin{rem}
    Note that the cubical degeneracy maps insert a copy of the zero vector. The resulting chains are quotiented out upon passing to the normalized chains. This reproduces a relation analogous to \eqref{equ:sh null rel} on sharblies.
\end{rem}

%The nerve differential is included using the standard Koszul sign convention.

\subsection{Comparison with the Waldhausen construction}
\label{sec:triangulation-waldhausen-comparison}

\newcommand{\sgn}{\mathrm{sgn}}
\begin{proposition}\label{prop:cubical-waldhausen-comparison}
\label{prop:deformed homology}
There is a zigzag of quasi-isomorphisms of chain complexes
\begin{equation}\label{eq:cubical-waldhausen-zigzag}
  (\shB,d+D)
  \xleftarrow{\ \pi\ }X
  \xrightarrow{\ \rho\ }C^{\mathrm{Wal}}_\bullet
\end{equation}
with $X$ as in \eqref{equ:X def}.
Consequently,
\[
  H_\bullet(\shB,d+D)
  \cong H_\bullet(BK(\Z);\Q).
\]
\end{proposition}
We first construct the two maps.
Let $F=(V,\mathbf v, \mathcal V)$ be an enhanced ordered configuration of length $k$. For
$\sigma\in S_k$, set $I_a^\sigma:=\{\sigma(1),\ldots,\sigma(a)\},
  \qquad 0\leq a\leq k$
and let 
\begin{equation}\label{eq:permutation-flag}
  \operatorname{Flag}_\sigma(F)_{a,b}
  :=F_{I_a^\sigma,I_b^\sigma}.
\end{equation}
This defines a functor $\operatorname{Flag}_\sigma \colon \mathcal E_k\to wS_k\mathcal P$.

Now on an element 
\[
x=(F_0 \xrightarrow{f_1} F_1  \xrightarrow{f_2} \cdots \xrightarrow{f_q} F_q) \in X_{k,q}
\]
we define our two maps by 
\begin{align}
\label{equ:rho def}
  \rho(x) &:= \sum_{\sigma\in S_k} (-1)^\sigma \operatorname{Flag}_\sigma(x), \\
  \label{equ:pi def}
  \pi(x) &:= 
  \begin{cases}
     [v_1,\dots,v_k] & \text{if $q=0$}\\
     0 &\text{otherwise}
  \end{cases},
\end{align}
where $v_1,\dots,v_k$ are the marked vectors in the enhanced ordered configuration $F_0$.

\begin{lemma}\label{lem:pi well}
    The map \eqref{equ:pi def} commutes with the differentials, 
    \[
    (d+D) \pi = \pi \partial_{\rm tot}.
    \]
\end{lemma}
\begin{proof}
    Recall that $\partial_{\rm tot}=\partial_\square + (-1)^k\partial_{\rm nerve}$. 
    It is clear that $\pi \partial_{\rm nerve}=0$ since $\shB$ is defined using $\GL$-coinvariants.
    Hence we may assume that $q=0$ and it remains to show that 
    \[
    (d+D) \pi x = \pi\partial_\square x
    = \sum_{j=1}^k (-1)^{j-1}(\pi R_jx-\pi C_jx).
    \]
    If $v_j$ (as above) is not a coloop, then 
    \[
    (-1)^{j-1}\pi R_j x = (-1)^{j-1}[v_1,\dots,\hat v_j ,\dots, v_k],
    \]
    which equals the $j$th term in $d\pi x$, see \eqref{equ:sh diff}.
    Similarly, $-C_jF$ gives the $j$th (quotient) term
    in $D_1$, 
    \[
    (-1)^{j}\pi C_j x = (-1)^{j} [\bar v_1,\dots,\hat {\bar v}_j ,\dots, \bar v_k]_{V/(\Z v_j)_{\sat}},
    \]
    see \eqref{equ:D formula}. 
    If $v_j$ is a coloop, the $j$th term in $d\pi F$ vanishes by the span relation \eqref{equ:sh span rel}. In this case $-\pi C_jF$ still provides the $j$th term of $D_1$, while $R_jF$ provides the corresponding $D_2$-term
    \[
    (-1)^{j-1}\pi R_j x = (-1)^{j-1}
    [v_1,\dots,\hat v_j,\dots,v_k]_{\vspan(v_1,\dots,\hat v_j,\dots,v_k)_{\sat}}.\qedhere
    \]
\end{proof}

\begin{lemma}\label{lem:rho well}
    The map \eqref{equ:rho def} commutes with the differentials, 
    \[
    \partial_{\rm Wal} \rho = \rho \partial_{\rm tot}.
    \]
\end{lemma}
\begin{proof}
It is clear that $\partial_{\rm nerve} \rho = \rho \partial_{\rm nerve}$, so it remains to check that 
\[
\partial_S \rho (x)  = \rho (\partial_\square x).
\]
Decompose $\partial_S = \sum_{j=0}^k (-1)^{j-1} \partial_j$ into the simplicial boundary maps $\partial_j$.
Concretely, for $1\leq j\leq k-1$ the operation $\partial_j$ forgets the $j$-th constituent of the flag, so that
\[
\partial_j \operatorname{Flag}_\sigma(x) = 
\partial_j \operatorname{Flag}_{\sigma\circ (j, j+1)}(x), 
\]
with $(j, j+1)\in S_k$ the pair transposition.
Hence the corresponding terms of $\partial_j \rho(x)$ in the alternating sum \eqref{equ:rho def} cancel. We obtain $\partial_S \rho (x)=-(\partial_0 + (-1)^k \partial_k) \rho(x)$.

The term $\partial_0\rho(x)$ is obtained by quotienting the ambient module by the first constituent of the flag. If $\sigma(1)=j$ then we have 
\[
\partial_0 \operatorname{Flag}_\sigma(x) 
=
\operatorname{Flag}_{\bar \sigma}(C_j x),
\]
where $\bar \sigma\in S_{k-1}$ is obtained from $\sigma$ by deleting $j$ and renumbering. We have $(-1)^\sigma = (-1)^{j-1} (-1)^{\bar \sigma}$.

The term $\partial_k\rho(x)$ is obtained by replacing the ambient module by the $k-1$st term of the flag. In particular,
if $\sigma(k)=i$ then 
\[
\partial_{k} \operatorname{Flag}_\sigma(x) 
=
\operatorname{Flag}_{\tilde \sigma}(R_i x),
\]
with $\tilde \sigma$ obtained by removing $i$ and renumbering,
so that $(-1)^\sigma = (-1)^{k-i} (-1)^{\tilde \sigma}$.
We therefore have
\begin{align*}
    \partial_S \rho(x) &= -(\partial_0 + (-1)^k \partial_k) \rho(x)
    \\&= - \sum_{\sigma\in S_k} (-1)^\sigma 
    (\operatorname{Flag}_{\bar \sigma}(C_jx) + (-1)^k \operatorname{Flag}_{\tilde \sigma}(R_ix))
    \\&= - \sum_{\sigma\in S_k} 
    (-1)^{j-1} (-1)^{\bar \sigma}
    \operatorname{Flag}_{\bar \sigma}(C_jx)
    +
    \sum_{\sigma\in S_k} 
    (-1)^{i-1} (-1)^{\tilde \sigma}
    \operatorname{Flag}_{\tilde \sigma}(R_ix)
    \\&=
    - \sum_{\sigma\in S_{k-1}} (-1)^{\sigma}
    \sum_{j=1}^k (-1)^{j-1} 
    \operatorname{Flag}_{\sigma}(C_jx)
    +
    \sum_{\sigma\in S_{k-1}} (-1)^{\sigma} \sum_{i=1}^k
    (-1)^{i-1} 
    \operatorname{Flag}_{\sigma}(R_ix)
    \\&=
    \rho\left( \sum_{j=1}^k (-1)^{j-1} (R_j-C_j)x \right)
    =\rho( \partial_\square x).
\end{align*}
In the second and third line $j=\sigma(1)$ (resp. $i=\sigma(k)$) depend on the permutation $\sigma$, though this is not indicated in the notation. 
\end{proof}

\begin{lemma}\label{lem:pi qiso}
    The map $\pi\colon(X,\partial_{\rm tot})\to (\shB,d+D)$ is a quasi-isomorphism.
\end{lemma}
\begin{proof}
    $X$ is a first quadrant bicomplex, and we may understand $\shB$ as a first quadrant bicomplex concentrated in degree 0 in the second coordinate, and $\pi$ is then a map of bicomplexes.
    Taking the associated spectral sequences, it suffices to show that for each fixed $k$, the morphism 
    \[
    \pi \colon (X_{k\bullet} = \bigoplus_{q\geq 0}X_{kq}, \partial_{\rm nerve} ) \longrightarrow (\shB_k, 0)  
    \]
    is a quasi-isomorphism.
    
    To study the left-hand side, first consider the full subgroupoid $\mathcal E_k^\times \subset \mathcal E_k$ whose objects are enhanced ordered configurations such that the marked vectors are nonzero. Denoting the normalized chain complex of its nerve by $C_\bullet(\mathcal E_k^\times)$ we have the isomorphism of chain complexes
    \[
    X_{k\bullet} \cong (C_\bullet(\mathcal E_k^\times) \otimes \sgn_k)_{S_k}.
    \]
    Forgetting the enhancement and remembering only the ambient module and the tuple of marked vectors produces from $\mathcal E_k^\times$ a groupoid $\bar {\mathcal E}_k^\times$, and the map $\mathcal E_k^\times\to \bar {\mathcal E}_k^\times$ is an equivalence of groupoids, so that we have a quasi-isomorphism 
    \[
    (C_\bullet(\mathcal E_k^\times) \otimes \sgn_k)_{S_k}
    \longrightarrow 
    (C_\bullet(\bar{\mathcal E}_k^\times) \otimes \sgn_k)_{S_k}.
    \]
    The automorphism groups of $\bar {\mathcal E}_k^\times$ are finite, since an automorphism has to fix the spanning (over $\Q$) tuple of vectors up to signs, and there are finitely many choices of signs. Hence the positive degree homology of $\bar {\mathcal E}_k^\times$ vanishes.
    Since the functor $(-\otimes \sgn_k)_{S_k}$ is exact over $\Q$ this also holds for the (signed) $S_k$-coinvariants of the normalized rational chain complex of the nerve of $\bar {\mathcal E}_k^\times$. Furthermore, the zeroth homology is evidently identified with $\shB_k$.
    We therefore have
    \begin{equation}\label{equ:X hom}
    H_q(X_{k\bullet}) \cong 
    H_q( (C_\bullet(\bar{\mathcal E}_k^\times) \otimes \sgn_k)_{S_k})
    \cong 
    \begin{cases}
        \shB_k & \text{for $q=0$} \\
        0 & \text{for $q>0$}
    \end{cases}.
    \end{equation}
    Inspecting the chain of morphisms above, we see that the homology isomorphism \eqref{equ:X hom} is the one induced by $\pi$, so that the lemma follows.
\end{proof}

\begin{lemma}\label{lem:rho qiso}
    The map $\rho\colon X\to C_\bullet^{\rm Wal}$ is a quasi-isomorphism.
\end{lemma}
\begin{proof}
Both $X$ and $C_\bullet^{\rm Wal}$ are equipped with a bounded below exhaustive filtration by the rank of the ambient module. The morphism $\rho$ respects the filtrations. Hence for showing that $\rho$ is a quasi-isomorphism it suffices to check that the associated graded morphism 
\[
\gr_n \rho \colon \gr_n X\longrightarrow \gr_n C_\bullet^{\rm Wal}
\]
is a quasi-isomorphism for every $n$.
This is the statement we will show. For $n=0$ both sides are $\Q$, so we may assume that $n>0$.

For $\mathcal G$ a groupoid, let us denote by $C_\bullet(\mathcal G)$ the normalized rational chain complex of its nerve, computing the rational groupoid homology. We denote its differential by $\partial_{\rm nerve}$.
Let $\mathcal D_{n,k}$ be the groupoid whose objects are pairs $(V,\mathbf v)$ consisting of a free $\Z$-module of rank $n$ and a spanning tuple of non-zero marked vectors $\mathbf v=(v_1,\dots,v_k)\in \overline V^k$.
The morphisms are isomorphisms of the ambient modules that map the marked vectors to each other. We define the complex 
\[
h\shB^{(n)} := \operatorname{Tot}_{k,q}(C_q(\mathcal D_{n,k})\otimes \sgn_k)_{S_k}
\]
with differential $\partial_{\Sh} + (-1)^k \partial_{\rm nerve}$, with $\partial_{\Sh}$ defined as (see also \eqref{equ:sh diff})
\begin{equation*}
\partial_{\Sh}
 ([v_1,\dots, v_{k}],f_1,\dots,f_q) = 
\sum_{\substack{ j=1 \\ v_j\neq 0 \\ v_j \text{ no coloop}}}^k (-1)^{j-1} ([v_1,\dots,\hat v_j,\dots, v_{k}],f_1,\dots,f_q).
\end{equation*}
There is an obvious forgetful map $a\colon\gr_nX \to h\shB^{(n)}$. It is a quasi-isomorphism, since in each fixed length $k$ (of the tuple of marked vectors) it is induced by an equivalence of groupoids, and taking $S_k$-coinvariants commutes with taking homology over $\Q$.

On the other hand, let $\mF_{n,p}$ be the groupoid of flags of saturated submodules 
\[
0=L_0 \subsetneq L_1 \subsetneq \cdots \subsetneq L_p =L
\]
with the ambient module $L$ of rank $n$. The morphisms are isomorphisms of the ambient module that preserve the flags.
Define 
\[
hS^{(n)}:= \operatorname{Tot}_{p,q} C_q(\mF_{n,p})
\]
with differential $\partial_S + (-1)^{p} \partial_{\rm nerve}$, 
\begin{multline*}
\partial_S (L_0 \subsetneq L_1 \subsetneq \cdots \subsetneq L_p,f_1,\dots,f_q)
=\\
\sum_{j=1}^{p-1}
(-1)^{j-1} (L_0 \subsetneq L_1 \subsetneq \cdots\subsetneq \hat L_j  \subsetneq\cdots \subsetneq L_p,f_1,\dots,f_q).
\end{multline*}
The complexes $h\shB^{(n)}$ and $hS^{(n)}$ model, respectively,
the homotopy $\GL(n,\Z)$-coinvariants of $\Sh(\Z^n)$ and 
$\hat N \T_\bullet(\Q^n)[-2]$ of Section \ref{sec:red aug}. Extending the map $\chi$ by the identity in the nerve direction we hence obtain a morphism $\chi_h\colon h\shB^{(n)}\to hS^{(n)}$
which is a quasi-isomorphism by Lemma \ref{lem:AMP chi}.

Let $W_n^{\rm str}\subset \gr_n C_\bullet^{\rm Wal}$ be the subcomplex spanned by nerve simplices of diagrams satisfying
\[
r_0:=0< r_1:=\rk A_{0,1} < \cdots < r_p:=\rk A_{0,p} =n.
\]
The inclusion $j\colon W_n^{\rm str}\to \gr_n C_\bullet^{\rm Wal}$ is a quasi-isomorphism by the following argument. We first take the nerve homology. The differential $\partial_{\rm nerve}$ leaves invariant the increasing sequence $(r_i)$ of ranks. Any simplex with some $r_i=r_{i+1}$ is isomorphic to a degenerate one. Hence on the $\partial_{\rm nerve}$-homology the parts corresponding to such rank sequences become zero, since we are taking normalized chains.

Furthermore, we have the forgetful morphism $b\colon W_n^{\rm str}\to hS^{(n)}$ remembering only the first line of the diagram. 
For each $p$ separately the forgetful functor is an equivalence of groupoids and therefore $b$ is a quasi-isomorphism.
The map $\gr_n \rho$ factors through a map $\rho^{\rm str}\colon \gr_nX \to W_n^{\rm str}$ since the image of $\gr_n \rho$ consists of flags with strictly increasing ranks. We hence obtain a commutative diagram 
\[
\begin{tikzcd}
    \gr_n X \ar{r}{\rho^{\rm str}} \ar{d}{a}[swap]{\sim}  \ar[bend left]{rr}{\gr_n\rho}
    & W_n^{\rm str} \ar{r}{j}[swap]{\sim} \ar{d}{b}[swap]{\sim}  & \gr_n C_\bullet^{\rm Wal} \\
    h\shB^{(n)}\ar{r}{\chi_h}[swap]{\sim} & hS^{(n)}
\end{tikzcd}.
\]
It readily follows that $\gr_n\rho$ is a quasi-isomorphism.
\end{proof}

\begin{proof}[Proof of Proposition \ref{prop:cubical-waldhausen-comparison}]
By Lemmas \ref{lem:pi well} and \ref{lem:rho well} the maps of \eqref{eq:cubical-waldhausen-zigzag} are morphisms of complexes, and by Lemmas \ref{lem:pi qiso} and \ref{lem:rho qiso} they are quasi-isomorphisms. 
The second statement of Proposition \ref{prop:cubical-waldhausen-comparison} follows immediately since $C^{\rm Wal}_\bullet$ computes $H_\bullet(BK(\Z);\Q)$.
\end{proof}

\subsection{End of proof of Theorem \ref{thm:main 2}}
It remains to check that $\Phi$ of \eqref{equ:Phi deformed} induces an isomorphism on homology.
Let us denote the homology map by 
\begin{equation}\label{equ:H Phi}
[\Phi]\colon H(U^c\G_0^+,d+\nabla) \longrightarrow H(\shB,d+D)
\end{equation}
and recall what we know so far.
First, $[\Phi]$ is a map of graded Hopf algebras.
We also already know that both sides are isomorphic to $H:=\Q[t,W_3,W_5,\dots]$. 
More precisely, we know that the natural map $H\to H(U^c\G_0^+,d+\nabla)$ is a Hopf algebra isomorphism, see Corollary \ref{cor:KWZ}. On the other hand, we know from Proposition \ref{prop:deformed homology} that $H(\shB,d+D)\cong H(BK(\Z);\Q)$, and by a classical computation of Borel \cite{Borel1974} the latter is identified with $H$. We emphasize that we will just need that there is some isomorphism of graded vector spaces 
$H\cong H(\shB,d+D)$, but we do not require any further properties or knowledge of the isomorphism.
Also, $H$ and hence both sides of \eqref{equ:H Phi} are non-negatively graded, connected and finite dimensional in each homological degree.

Now, for $[\Phi]$ to be an isomorphism it is sufficient to show that $[\Phi]$ is injective, since both source and target have the same finite dimension degree-wise.
We will show injectivity by induction on the degree.
Trivially, $[\Phi]$ is injective in homological degree zero.
Suppose $[\Phi]$ is injective in degrees $< n$ and suppose there was some $x\in H(U^c\G_0^+,d+\nabla)$ of degree $n$ so that $[\Phi](x)=0$.
Let $\bar\Delta$ denote the reduced coproduct.
Then by compatibility with the Hopf structure 
\[
([\Phi]\otimes [\Phi])(\bar \Delta x)=
\bar \Delta([\Phi](x))=0.
\]
Hence, since $[\Phi]$ is injective in degrees $<n$ and only such lower degree pieces are used on the left-hand side of the equation above, we see that $\bar \Delta x=0$. Hence $x$ is primitive. 
But by Corollary \ref{cor:Wn def injective} the map $[\Phi]$ is injective on the subspace $\vspan(t,W_3,W_5,\dots)\subset H$ of primitive elements.
Hence $x=0$.
This shows that $[\Phi]$ is injective, hence an isomorphism.
\hfill\qed

\section{Remark on the odd graph complex}
Many of the above constructions have analogues for the homology $H_\bullet(\GL(n,\Z),\St_n\otimes \Q\otimes \det)$ with values in the Steinberg module times the sign representation.
For example, we may define 
\begin{align*}
\shA^{\det}_d &= \bigoplus_n H_{d}(\GL(n,\Z),\St_n\otimes \Q\otimes \det)
&
\shB^{\det}_d &= \bigoplus_n \left( \Sh(\Z^n)_{d+n} \otimes \det \right)_{\GL(n,\Z)}.
\end{align*}
Then $H(\shB^{\det})=\shA^{\det}$. Furthermore, $\shB^{\det}$ is a dg Hopf algebra as shown in \cite{AMP}, and hence $\shA^{\det}$ is a graded Hopf algebra.

Similarly, on the graph complex side we may define the ``odd'' graph complex $\GC_1$ whose elements are $\Q$-linear series of pairs $(\Gamma, o)$ with the orientation datum $o$ a choice of one of the two possible generators of the $\Z$-module 
\[
\det( E(\Gamma)) \otimes \det(H^1(\Gamma;\Z)),
\]
with the relations, differential and Lie bracket naturally extended, and the degree being the number of vertices minus 1.
The pre-dual dg Lie coalgebra we denote by $\G_1$.
Then by a straightforward extension of the definitions above we obtain a map of dg Hopf algebras 
\[
\Phi^{\det} \colon U^c\G_1 \longrightarrow \shB^{\det}.
\]
Unfortunately, this latter morphism  seems to be fairly trivial on homology as far as the homology tables are known.
Hence we refrain from discussing this ``odd'' case further and leave its study to future work.

\bibliographystyle{alpha}
\bibliography{refs}

\begin{thebibliography}{DSEVKM26}

\bibitem[AGM12]{AGMC}
Avner Ash, Paul~E. Gunnells, and Mark McConnell.
\newblock Resolutions of the {Steinberg} module for {$\mathrm{GL}(n)$}.
\newblock {\em Journal of Algebra}, 349(1):380--390, 2012.

\bibitem[AMP]{AMP}
Avner Ash, Jeremy Miller, and Peter Patzt.
\newblock {Hopf algebras, Steinberg modules, and the unstable cohomology of
  $SL_n(\mathbb Z)$ and $GL_n(\mathbb Z)$}.
\newblock Preprint.
\newblock \href{https://arxiv.org/abs/2404.13776}{arXiv:2404.13776}.

\bibitem[BCGP]{BCGP}
Francis Brown, Melody Chan, S{\o}ren Galatius, and Sam Payne.
\newblock {Hopf algebras in the cohomology of $\mathcal{A}_g$,
  $\mathrm{GL}_n(\mathbb{Z})$, and $\mathrm{SL}_n(\mathbb{Z})$}.
\newblock Preprint.
\newblock \href{https://arxiv.org/abs/2405.11528}{arXiv:2405.11528}.

\bibitem[Bor74]{Borel1974}
Armand Borel.
\newblock Stable real cohomology of arithmetic groups.
\newblock {\em Ann. Sci. \'Ecole Norm. Sup. (4)}, 7(2):235--272, 1974.

\bibitem[BPP]{BrownPanzerPortner}
Francis Brown, Erik Panzer, and Jean-Luc Portner.
\newblock {The motivic Lie algebra embeds into the cohomology of the general
  linear group}.
\newblock Preprint.
\newblock \href{https://arxiv.org/abs/2607.29671}{arXiv:2607.29671}.

\bibitem[Bro12]{Brown}
Francis Brown.
\newblock Mixed {Tate} motives over {$\mathbb{Z}$}.
\newblock {\em Annals of Mathematics}, 175(2):949--976, 2012.

\bibitem[BS73]{BorelSerre}
A.~Borel and J.-P. Serre.
\newblock Corners and arithmetic groups.
\newblock {\em Commentarii Mathematici Helvetici}, 48(1):436--491, 1973.

\bibitem[BS25]{BrownSchnetz}
Francis Brown and Oliver Schnetz.
\newblock The wheel classes in the locally finite homology of
  {$\mathrm{GL}_{n}(\mathbb{Z})$}, canonical integrals and zeta values.
\newblock {\em Geometry \& Topology}, 29(8):4389--4447, 2025.

\bibitem[Dri90]{Drinfeld}
V.~G. Drinfeld.
\newblock On quasitriangular quasi-{Hopf} algebras and on a group that is
  closely connected with {$\mathrm{Gal}(\overline{\mathbb Q}/\mathbb Q)$}.
\newblock {\em Algebra i Analiz}, 2(4):149--181, 1990.
\newblock English translation in Leningrad Math. J. 2 (1991), no. 4, 829--860.

\bibitem[DSEVKM26]{DEKM}
Mathieu Dutour~Sikiri{\'c}, Philippe Elbaz-Vincent, Alexander Kupers, and
  Jacques Martinet.
\newblock Voronoi complexes in higher dimensions, cohomology of
  {{\(GL_N(\mathbb{Z})\)}} for {{\(N\geqslant 8\)}} and the triviality of
  {{\(K_8(\mathbb{Z})\)}}.
\newblock {\em J. Inst. Math. Jussieu}, 25(1):565--590, 2026.

\bibitem[Gon05]{Goncharov}
A.~B. Goncharov.
\newblock {Polylogarithms, regulators, and Arakelov motivic complexes}.
\newblock {\em Journal of the American Mathematical Society}, 18(1):1--60,
  2005.

\bibitem[Gon26]{GoncharovPersonal}
A.~B. Goncharov.
\newblock Personal communication, 2026.

\bibitem[KW{\v Z}17]{KWZ}
Anton Khoroshkin, Thomas Willwacher, and Marko {\v Z}ivkovi{\'c}.
\newblock {Differentials on graph complexes}.
\newblock {\em Advances in Mathematics}, 307:1184--1214, 2017.

\bibitem[LS76]{LeeSzczarba}
Ronnie Lee and R.~H. Szczarba.
\newblock On the homology and cohomology of congruence subgroups.
\newblock {\em Inventiones mathematicae}, 33(1):15--53, 1976.

\bibitem[OG08]{OG}
J.-M. Oudom and D.~Guin.
\newblock {On the Lie enveloping algebra of a pre-Lie algebra}.
\newblock {\em Journal of K-Theory}, 2(1):147--167, 2008.

\bibitem[PP26]{PalPayne}
Urshita Pal and Sam Payne.
\newblock The {Q}uillen and sharbly {H}opf algebra structures on {S}teinberg
  homology coincide.
\newblock Preprint, 2026.
\newblock \href{https://arxiv.org/abs/2609.38735}{arXiv:2609.38735}.

\bibitem[Qui73]{QuillenSpSq}
Daniel Quillen.
\newblock Finite generation of the groups {{\(K_i\)}} of rings of algebraic
  integers.
\newblock Algebr. {{\(K\)}}-{Theory} {I}, {Proc}. {Conf}. {Battelle} {Inst}.
  1972, {Lect}. {Notes} {Math}. 341, 179-198 (1973)., 1973.

\bibitem[RW]{RossiWillwacher}
Carlo~A. Rossi and Thomas Willwacher.
\newblock {P. Etingof's conjecture about Drinfeld associators}.
\newblock Preprint.
\newblock \href{https://arxiv.org/abs/1404.2047}{arXiv:1404.2047}.

\bibitem[Sol69]{Solomon}
Louis Solomon.
\newblock The {Steinberg} character of a finite group with {BN}-pair.
\newblock In {\em Theory of Finite Groups (Symposium, Harvard Univ., Cambridge,
  Mass., 1968)}, pages 213--221. Benjamin, New York, 1969.

\bibitem[Wal85]{Waldhausen}
Friedhelm Waldhausen.
\newblock Algebraic {$K$}-theory of spaces.
\newblock In Andrew Ranicki, Norman Levitt, and Frank Quinn, editors, {\em
  Algebraic and Geometric Topology}, volume 1126 of {\em Lecture Notes in
  Mathematics}, pages 318--419, Berlin, Heidelberg, 1985. Springer Berlin
  Heidelberg.

\bibitem[War23]{Ward21}
Benjamin~C. Ward.
\newblock Wheel graph homology classes via {Lie} graph homology.
\newblock {\em Journal of Noncommutative Geometry}, 17(2):693--717, 2023.

\bibitem[Wila]{WillwacherFreeness}
Thomas Willwacher.
\newblock {On free generators in the Grothendieck-Teichm\"uller Lie Algebra}.
\newblock Preprint.
\newblock \href{https://arxiv.org/abs/2607.18793}{arXiv:2607.18793}.

\bibitem[Wilb]{WillwacherTriconnected}
Thomas Willwacher.
\newblock {The triconnected Kontsevich graph complex}.
\newblock Preprint.
\newblock \href{https://arxiv.org/abs/2503.17131}{arXiv:2503.17131}.

\bibitem[Wil15]{WillwacherGrt}
Thomas Willwacher.
\newblock {M. Kontsevich's graph complex and the Grothendieck--Teichm{\"u}ller
  Lie algebra}.
\newblock {\em Inventiones mathematicae}, 200(3):671--760, 2015.

\end{thebibliography}

\end{document}